\documentclass[12pt]{amsart}
\usepackage{euscript,epsf,amssymb,xr}

\let\oldmarginpar\marginpar
\renewcommand\marginpar[1]
{\oldmarginpar{\tiny\bf \begin{flushleft} #1 \end{flushleft}}}

\input xy
\xyoption{all}

\newcommand{\la}{\langle}
\newcommand{\ra}{\rangle}

\newtheorem{theorem}{\bf Theorem}[section]
\newtheorem{lemma}[theorem]{\bf Lemma}

\newtheorem{remark}[theorem]{\bf Remark}

\newtheorem{definition}[theorem]{\bf Definition}

\renewcommand{\AA}{{\Bbb A}}

\newcommand{\CC}{{\Bbb C}}

\newcommand{\EE}{{\Bbb E}}

\newcommand{\II}{{\Bbb I}}

\newcommand{\NN}{{\Bbb N}}
\newcommand{\OO}{{\Bbb O}}

\newcommand{\RR}{{\Bbb R}}

\newcommand{\VV}{{\Bbb V}}

\newcommand{\ZZ}{{\Bbb Z}}

\newcommand{\Gg}{{\mathfrak G}}

\newcommand{\Xx}{{\mathfrak X}}

\newcommand{\ggreat}{>\kern-.7ex>}
\newcommand{\ssmall}{<\kern-.7ex<}
\newcommand{\qu}{/\kern-.7ex/}
\newcommand{\exh}{\to\kern-1.8ex\to}

\newcommand{\aA}{{\EuScript{A}}}

\newcommand{\cC}{{\EuScript{C}}}

\newcommand{\eE}{{\EuScript{E}}}
\newcommand{\fF}{{\EuScript{F}}}
\newcommand{\gG}{{\EuScript{G}}}

\newcommand{\mM}{{\EuScript{M}\mathrm{et}}}
\newcommand{\MMM}{{\EuScript{M}}}
\newcommand{\nN}{{\EuScript{N}}}
\newcommand{\oO}{{\EuScript{O}}}

\newcommand{\uU}{{\EuScript{U}}}
\newcommand{\vV}{{\EuScript{V}}}

\newcommand{\xX}{{\EuScript{X}}}
\newcommand{\yY}{{\EuScript{Y}}}

\newcommand{\GL}{\operatorname{GL}}

\newcommand{\Aut}{\operatorname{Aut}}

\newcommand{\Crit}{\operatorname{Crit}}

\newcommand{\Diff}{\operatorname{Diff}}
\newcommand{\End}{\operatorname{End}}

\newcommand{\Id}{\operatorname{Id}}
\newcommand{\Ind}{\operatorname{Ind}}

\newcommand{\Ker}{\operatorname{Ker}}
\newcommand{\Lie}{\operatorname{Lie}}
\newcommand{\Map}{\operatorname{Map}}

\newcommand{\supp}{\operatorname{supp}}

\newcommand{\reg}{\operatorname{reg}}

\newcommand{\ov}{\overline}

\newcommand{\wh}{\widehat}
\newcommand{\wt}{\widetilde}

\title[Automorphisms of generic gradient vector fields with symmetries]
{Automorphisms of generic gradient vector fields with prescribed finite symmetries}

\author{Ignasi Mundet i Riera}
\address{Facultat de Matem\`atiques i Inform\`atica\\
Universitat de Barcelona\\
Gran Via de les Corts Catalanes 585\\
08007 Barcelona \\
Spain}
\email{ignasi.mundet@ub.edu}

\date{October 30, 2017}

\thanks{This work has been partially supported by the (Spanish) MEC Project MTM2015-65361-P MINECO/FEDER, UE}

\begin{document}

\maketitle

\begin{abstract}
Let $M$ be a compact and connected smooth manifold endowed with a smooth
action of a finite group $\Gamma$, and let $f$ be a
$\Gamma$-invariant Morse function on $M$. We prove that the space
of $\Gamma$-invariant Riemannian metrics on $M$ contains a
residual subset $\mM_f$ with the following property. Let
$g\in\mM_f$ and let $\nabla^gf$ be the gradient vector field of $f$
with respect to $g$. For any diffeomorphism $\phi\in\Diff(M)$ preserving
$\nabla^gf$ there exists some $t\in\RR$ and some
$\gamma\in\Gamma$ such that for every $x\in M$ we have
$\phi(x)=\gamma\,\Phi_t^g(x)$, where $\Phi_t^g$ is the time-$t$
flow of the vector field $\nabla^gf$.
\end{abstract}

\tableofcontents

\section{Introduction}
Define the automorphism group of a vector field on a smooth manifold to be
the group of diffeomorphisms of the manifold preserving the vector field.
A natural question is how small the automorphism group of a vector field
can be. Suppose that we only consider vector fields which are invariant
under a fixed finite group action on the manifold.
In this situation, the automorphism group of the vector field always includes
the action of the finite group and the flow of the vector field.
Our main result implies that if the manifold is compact
and connected then
the set of invariant gradient vector fields whose automorphism
group contains nothing more than this is residual in the set of all invariant gradient vector fields.
(Recall that a residual subset is a countable intersection of
dense open subsets. Since the space of smooth invariant
gradient vector fields is Baire\footnote{See
e.g. \cite[Chap. 2, Theorem 4.4]{H} and the comments afterwards.}, any of its
residual subsets is dense.)

Let us explain in more concrete terms our motivation and main result.

Take $M$ to be a smooth (=$\cC^{\infty}$) manifold, and denote
by $\Xx(M)$ the vector space of smooth vector fields on $M$,
endowed with the $\cC^{\infty}$ topology. Denote the
automorphism group of a vector field $\xX\in\Xx(M)$ by
$$\Aut(\xX)=\{\phi\in\Diff(M)\mid \phi_*\xX=\xX\}.$$
If $\xX\neq 0$ then $\Aut(\xX)$ contains a central subgroup
isomorphic to $\RR$, namely the flow generated by $\xX$. Denote
by $\Aut(\xX)/\RR$ the quotient of $\Aut(\xX)$ be this
subgroup.

Suppose that $M$ is endowed with a smooth and effective action
of a finite group $\Gamma$.  Let $\Xx(M)^{\Gamma}\subset\Xx(M)$
be the space of $\Gamma$-invariant vector fields on $M$. F.J.
Turiel and A. Viruel proved recently in \cite{TV} that there
exists some $\xX\in\Xx(M)^{\Gamma}$ such that
$\Aut(\xX)/\RR\simeq \Gamma$. The vector field $\xX$ is given
explicitly in \cite{TV} as a gradient vector field for a
carefully constructed Morse function and a suitable Riemannian
metric. One may wonder, in view of that result, whether the set
of $\xX\in\Xx(M)^{\Gamma}$ satisfying $\Aut(\xX)/\RR\simeq
\Gamma$ is generic in some sense, at least if one restricts to
some particular family of vector fields (such as gradient
vector fields, for example). Here we give an affirmative answer
to this question assuming $M$ is compact and connected.

Suppose, for the rest of the paper, that $M$ is compact and connected. For
any function $f\in\cC^{\infty}(M)$ and any Riemannian metric
$g$ on $M$ we denote by $\nabla^gf\in\Xx(M)$ the gradient of
$f$ with respect to $g$, defined by the condition that
$g(\nabla^gf,v)=df(v)$ for every $v\in TM$. The following is
our main result.

\begin{theorem}
\label{thm:main}
Let $\mM\subset\cC^{\infty}(M,S^2T^*M)$ denote the space of
Riemannian metrics on $M$, with the $\cC^{\infty}$ topology,
and let $\mM^{\Gamma}\subset\mM$ denote the space of
$\Gamma$-invariant metrics. Let $f$ be a $\Gamma$-invariant
Morse function on $M$. There exists a residual subset
$\mM_f\subset\mM^{\Gamma}$ such that any $g\in\mM_f$ satisfies
$\Aut(\nabla^gf)/\RR\simeq \Gamma$.
\end{theorem}

By a result of Wasserman \cite[Lemma 4.8]{W}, the space of
$\Gamma$-invariant Morse functions on $M$ is open and dense
within the space of all $\Gamma$-invariant smooth functions on
$M$ (for openness see the comments before Lemma 4.8 in
\cite{W}). Combining this result with Theorem \ref{thm:main} it
follows that if $\Gg(M)^{\Gamma}\subset\Xx(M)^{\Gamma}$ denotes
the set of $\Gamma$-invariant gradient vector fields, then the
set of $\xX\in\Gg(M)^{\Gamma}$ satisfying
$\Aut(\xX)/\RR\simeq\Gamma$ contains a residual subset in
$\Gg(M)^{\Gamma}$.

Probably Theorem \ref{thm:main} can be proved as well for most
proper $\Gamma$-invariant Morse functions on open manifolds
endowed with an smooth effective action of a finite group.
However, there are exceptions: if $M=\RR^n$ with $n>1$, and
$\Gamma$ is a finite group acting linearly on $M$ preserving
the standard euclidean norm, then $f:M\to\RR$, $f(v)=\|v\|^2$,
is a $\Gamma$-invariant Morse function, and for any
$\Gamma$-invariant Riemannian metric $g$ on $M$,
$\Aut(\nabla^gf)/\RR$ is bigger than $\Gamma$. This follows
from a result of Sternberg, see Theorem \ref{thm:Sternberg}
below.

For the case $\dim M=1$ (i.e., when $M$ is the circle) we prove
a stronger form of Theorem \ref{thm:main}, where residual is
replaced by open and dense, see Theorem \ref{thm:circle}. It is
not inconceivable that this can be done in all dimensions:
while the author has not managed to do so, he does not have any
reason to suspect that it might be false. In fact, a theorem of
Palis and Yoccoz answering an analogous question (in the
non-equivariant setting), with the set of gradient vector fields $\xX(M)$
replaced by a certain set of diffeomorphisms may suggest that it is true.
To be precise, let $\aA_1(M)\subset\Diff(M)$ be the (open) set of
Axiom A diffeomorphisms satisfying the transversality condition and
having a sink or a source. Palis and Yoccoz prove in \cite{PY1} that
the set of diffeomorphisms with the smallest possible centralizer
contains a $\cC^{\infty}$ open and dense subset of $\aA_1(M)$.
Note that the set $\aA_1(M)$ includes Morse--Smale diffeomorphisms, which are
analogues for diffeomorphisms of gradient vector fields.
However, if one considers the same question for the entire
diffeomorphism group endowed with the $\cC^1$ topology then openness
never holds, see \cite{BCW2}.

Before explaining the main ideas in the proof of Theorem
\ref{thm:main}, let us discuss some related problems. A natural
question is whether the property proved in Theorem
\ref{thm:main} is true when replacing the space of invariant
gradient vector fields by the entire space of invariant vector
fields.

\noindent{\bf Problem A.}
Does the set of $\xX\in\Xx(M)^{\Gamma}$ satisfying $\Aut(\xX)/\RR\simeq\Gamma$
contain a residual subset of $\Xx(M)^{\Gamma}$?

Problem A includes Theorem \ref{thm:main} as a particular case,
since $\Gg(M)^{\Gamma}$ is open in $\Xx(M)^{\Gamma}$. Note that
the case $\Gamma=\{1\}$ of Problem A (or even Theorem
\ref{thm:main}) is far from being trivial.

Define the centralizer $Z(\xX)$ of a vector field $\xX$ on $M$
to be the group of diffeomorphisms of $M$ that send orbits of
$\xX$ onto orbits of $\xX$. For any $\xX$, $\Aut(\xX)$ is a
subgroup of $Z(\xX)$, and one may try to explore analogues of
Theorem \ref{thm:main} and Problem A for $Z(\xX)$. However,
even the right question to ask is not clear in this situation.
P.R. Sad \cite{S} studied the case $\Gamma=\{1\}$. His main
result is that for a compact $M$ there is an open and dense
subset $\aA'$ of the set of Morse--Smale vector fields
$\aA\subset\Xx(M)$ such that for any $\xX\in\aA'$ there is a
neighborhood $V\subset\Diff(M)$ of the identity with the
property that any $\phi\in V\cap Z(\xX)$ preserves the orbits
of $\xX$. Unfortunately the restriction to a
neighborhood of the identity in $\Diff(M)$ can not be removed,
as Sad shows with an example.

It is natural to consider analogues of the previous problems replacing vector fields
by diffeomorphisms.
Define the automorphism group of a diffeomorphism $\phi\in\Diff(M)$ to be its
centralizer, i.e.,
$\Aut(\phi)=\{\psi\in\Diff(M)\mid \phi\psi=\psi\phi\}$.
Then $\la\phi\ra=\{\phi^k\mid k\in\ZZ\}$ is a central subgroup of $\Aut(\phi)$.
Let
$$\Diff^{\Gamma}(M)=\{\phi\in\Diff(M)\mid \phi\text{ commutes with the action of $\Gamma$}\}.$$

\noindent{\bf Problem B.}
Does the set of $\phi\in\Diff^{\Gamma}(M)$ such that
$\Aut(\phi)/\la\phi\ra\simeq\Gamma$ contain a residual subset
of $\Diff^{\Gamma}(M)$?

Of course, a positive answer to Problem B does not imply a
positive answer to Problem A, since a diffeomorphism $\phi$
such that $\Aut(\phi)/\la\phi\ra=\Gamma$ can not possibly
belong to the flow of a vector field (for otherwise
$\Aut(\phi)$ should contain a subgroup isomorphic to $\RR$).

One may consider restricted versions of Problem B involving particular diffeomorphisms,
for example, equivariant Morse--Smale diffeomorphisms \cite{field}. These are very particular diffeomorphisms,
but Problem B is already substantially nontrivial for them (even in the case $\Gamma=\{1\}$,
see below).

Problems A and B admit variations in which the regularity of
the vector fields or the diffeomorphisms is relaxed from
$\cC^{\infty}$ to $\cC^r$ for finite $r$. One can also consider
stronger questions replacing {\it residual} by {\it open and
dense} or weaker ones replacing {\it residual} by {\it dense}.

The case $\Gamma=\{1\}$ of Problem B is a famous question of
Smale. It appeared for the first time in \cite[Part IV, Problem (1.1)]{S0},
in more elaborate form in \cite{S1}, and it was included in his list of 18 problems
for the present century \cite{S2}. It was solved for
Morse--Smale $\cC^1$-diffeomorphisms by Togawa \cite{T} and
very recently for arbitrary $\cC^1$-diffeomorphisms by C.
Bonatti, S. Crovisier, A. Wilkinson in \cite{BCW} (see the
survey \cite{BCW1} for further references). The analogous
problem for higher regularity diffeomorphisms is open at
present, although there are by now plenty of partial results:
see e.g. \cite{K} for the case of the circle, \cite{PY1} for
elements in the set $\aA_1(M)$ defined above,
and \cite{PY2} for Anosov diffeomorphisms of tori.

Theorem \ref{thm:main} may be compared to similar results for
other types of tensors. For example, it has been proved in
\cite{mounoud} that on a compact manifold the set of metrics of
fixed signature with trivial isometry group is open and dense
in the space of all such metrics (see also \cite{CMR} for an
infinitesimal version of this with the compactness condition
removed).

\subsection{Main ideas of the proof}
To prove Theorem \ref{thm:main} we treat separately the cases
$\dim M=1$ and $\dim M>1$.
The case $\dim M=1$ is addressed in Section \ref{s:circle},
using rather ad hoc methods. An interesting ingredient is
an invariant of vector fields which, when nonzero,
distinguishes changes of orientation, and which plays an
important role in the classification up to diffeomorphisms of
vector fields on $S^1$ with nondegenerate zeroes.

The main ingredient in the case $\dim M>1$, common to other
papers addressing similar problems, is
a theorem of Sternberg \cite{St} on linearisation of
vector fields near sinks and sources, assuming there are no
resonances. The use of this result in this kind of problems
goes back to work of Kopell \cite{K}, and appears in papers of
Anderson \cite{A1} and Palis and Yoccoz \cite{PY1} among
others.
To apply this theorem in our situation we need to generalize
it to the equivariant setting under the presence of finite
symmetries. This poses some difficulties. For example, in
the equivariant case we can not suppose that the eigenvalues of
the linearisation of a generic vector field at fixed points are
all different: high multiplicities can not be avoided; in
particular, the centraliser of the linearisation is not
necessarily abelian (both Anderson \cite{A1} and Palis--Yoccoz
\cite{PY1} restrict themselves to the case in which the
eigenvalues are different). This is relevant for example when
extending the version of Sternberg's theorem for families
proved by Anderson to the equivariant setting (see Section
\ref{s:sternberg} for details on this).

We close this subsection with a more concrete description of the
proof of the case $\dim M>1$. Suppose a $\Gamma$-invariant Morse function $f$ has been chosen.
The set of metrics $\mM_f$ is defined as the intersection
of a set of invariant metrics, $\mM_0$, and a
countable sequence of subsets $\{\mM_{1,K}\}_{K\in\NN}$. Each of these sets is open and dense in $\mM^{\Gamma}$.

The metrics $g\in \mM_0$, defined in Subsection \ref{ss:mM-0},
have two properties: (1) the eigenvalues of the differential of
$\nabla^gf$ at each critical point are as much different among
themselves as they can be (in particular, the collection of
eigenvalues at two critical points coincide if and only if the
two points belong to the same $\Gamma$-orbit), and (2) there
are no resonances among eigenvalues at any critical point. The
second property allows us to use Sternberg's theorem on
linearisation on neighborhoods of sinks and sources, and a
theorem of Kopell which limits enormously the automorphisms of
the gradient vector field restricted to (un)stable manifolds of
sinks/sources.

The metrics $g\in\mM_{1,K}$, defined in Subsection \ref{s:mM-2-K}, have the following property.
Suppose that  $p$ is a sink and $W_g^s(p)$ is the stable manifold of $p$ for $\nabla^gf$, and that $q$ is a source and $W_g^u(q)$ is its unstable manifold for $\nabla^gf$. If $W_g^s(p)\cap W_g^u(q)$ is nonempty, then any automorphism of $\nabla^gf|_{W_g^s(p)}$  whose derivative at $p$ is at distance $\leq K$ from the identity and which matches on $W_g^s(p)\cap W_g^u(q)$
with an automorphism of $\nabla^gf|_{W_g^u(q)}$ is at distance $<K^{-1}$ from an automorphism coming from the action of $\Gamma$ and the flow of $\nabla^gf$.

After defining these sets of metrics, in Section \ref{s:proof-thm:main} we prove
Theorem \ref{thm:main} for manifolds of dimension greater than one, showing that if $g\in\mM_f$ then
$\Aut(\nabla^gf)/\RR\simeq\Gamma$.

The paper concludes with two appendices. The first one gives the proof of a technical result on the variation of the gradient flow of $\nabla^gf$ with respect to variations of $g$, and the second one contains a glossary of the notation used to address the case $\dim M>1$
(Section \ref{s:sternberg} and the next ones).

\noindent{\bf Acknowledgement.} I am very grateful to the
referee for pointing out a substantial simplification in the
proof of the main theorem, which was originally much longer and
more involved, and for detecting a number of mistakes and
suggesting improvements.

\section{Proof of Theorem \ref{thm:main} for $\dim M=1$}
\label{s:circle}

In this section we prove a strengthening of
the case $\dim M=1$ of Theorem \ref{thm:main}. More concretely, in Subsection
\ref{ss:proof-thm:circle} below we prove the following.

\begin{theorem}
\label{thm:circle}
Suppose that a finite group $\Gamma$ acts smoothly and effectively on $S^1$. Let $f$ be a $\Gamma$-equivariant Morse function on $S^1$. Let $\mM^{\Gamma}$ denote the set of $\Gamma$-invariant Riemannian metrics on
$S^1$, endowed with the $\cC^{\infty}$ topology. There exists a dense and open subset $\mM_f\subset\mM^{\Gamma}$ such that for every $g\in\mM_f$
we have $\Aut(\nabla^gf)/\RR\simeq\Gamma$.
\end{theorem}

\subsection{Classifying nondegenerate vector fields on the circle}
To prove Theorem \ref{thm:circle} we will need, in the case
when $\Gamma$ is generated by a rotation, an invariant of
nondegenerate vector fields on the circle that detects change
of orientations. This invariant is one of the ingredients of
the classification of nondegenerate vector fields on the circle
up to orientation preserving diffeomorphism. Detailed
expositions of this classification (in the broader context of
vector fields with zeroes of finite order) have appeared in
\cite{By,Hi}. Here we briefly explain the main ideas of this
result, focusing on the definition of the invariant, both for
completeness and to set the notation for later use.

For any $t\in\RR$ and  vector field $\xX$ we denote by $\Phi_t^{\xX}\in\Aut(\xX)$ the flow of $\xX$
at time $t$.

We first consider the local classification of vector fields with a nondegenerate zero.
For any nonzero real number $\lambda$ we denote by $\fF_{\lambda}$ the set of germs of
vector fields on a neighborhood of $0$ in $\RR$ of the form $h\,\partial_x$, where
$h(0)=0$ and $h'(0)=\lambda$ and $x$ is the standard coordinate in $\RR$. Let $\gG$ denote
the group of germs of diffeomorphisms of neighborhoods of $0$ in $\RR$.
For any $\xX\in\fF_{\lambda}$ we denote by $\Aut(\xX)$ the group of all $\phi\in\gG$ such that
$\phi_*\xX=\xX$. For example, $\Phi_t^{\xX}\in\Aut(\xX)$ for every $t$.
The proof of the next lemma follows from a straightforward computation and Cauchy's theorem
on ODE's.

\begin{lemma}
\label{lemma:local-models} Let $\lambda,\mu$ be nonzero real numbers.
\begin{enumerate}
\item Given $\xX\in\fF_{\lambda}$ and $\yY\in\fF_{\mu}$ there exists some $\phi\in\gG$
satisfying $\phi_*\xX=\yY$ if and only if $\lambda=\mu$.
\item For any $\lambda$ and $\xX\in\fF_{\lambda}$ the map $D:\Aut(\xX)\to\RR^*$ sending
$\phi$ to $\phi'(0)$ is an isomorphism of groups. Furthermore, $D\Phi_t^{\xX}=e^{\lambda t}$.
\end{enumerate}
\end{lemma}

We mention in passing that to prove the case $\dim M>1$ of Theorem \ref{thm:main} we will need to extend the previous lemma to higher dimensions, in a way equivariant with respect to finite group actions. This extension will be based on non-equivariant higher dimensional analogues of statements (1) and (2), which are respectively a theorem of
Sternberg (see \cite{St} and Theorem \ref{thm:Sternberg} below) and a theorem of Kopell (see \cite{K} and Subsection \ref{ss:mM-0} below). Both results are substantially deeper than Lemma \ref{lemma:local-models}, and
in particular they require a condition of non-resonance which is trivial in the one dimensional case.

We next explain the classification of nondegenerate vector fields on the
circle. We identify $S^1$ with $\RR/2\pi\ZZ$, so vector fields on $S^1$
can be written as
$$\xX=h\,\partial_x$$
where $h$ is a $2\pi$-periodic smooth function.
We say that $\xX$ is nondegenerate if $h(y)=0$ implies $h'(y)\neq 0$ ($h'(y)$ can be identified
with the derivative of $\xX$ at $y\in h^{-1}(0)$). An immediate consequence
is that $h$ contains finitely many zeroes in $[0,2\pi)$. Another consequence is that $h$
changes sign when crossing any zero of $h$, and this implies that $h^{-1}(0)$ contains
an even number of elements in $[0,2\pi)$.
To classify nondegenerate vector fields on the circle we will associate to them the number of
zeroes, their derivatives at the zeroes (up to cyclic order), and a global invariant
denoted by $\chi$.

To define $\chi$ suppose first of all that $h$ has no zeroes. Denoting by $\Phi_t^{\xX}$
the flow of $\xX$ seen as a vector field on $\RR$, there is a unique real number $t$ such
that $\Phi_t^{\xX}(y)=y+2\pi$ for every $y\in\RR$. Then we set
$$\chi(\xX):=t.$$

Now suppose that $h$ vanishes somewhere, and write its zeroes contained in $[0,2\pi)$ as
$$0\leq z_1<z_2<\dots<z_{2r}<2\pi.$$
We extend this finite list to an infinite sequence by setting $z_{i+2r}=z_i$ for every integer
$i$. Below, we implicitly consider similar periodic extensions for all objects that we are
going to associate to the zeroes $z_i$.
By (2) in Lemma \ref{lemma:local-models}, for every $i$ there exists a connected neighborhood
$U_i$ of $z_i$, disjoint from $z_{i-1}$ and $z_{i+1}$, and a unique smooth involution $\sigma_i:U_i\to U_i$ such that
\begin{equation}
\label{eq:prop-sigma}
\sigma_i(z_i)=z_i,\qquad \sigma_i'(z_i)=-1,\qquad
(\sigma_i)_*\xX=\xX.
\end{equation}
Choose for every $i$ some $t_i^+>z_i$
contained in $U_i$ and define $t_i^-=\sigma_i(t_i^+)$. Then we
have $t_i^+,t_{i+1}^-\in (z_i,z_{i+1})$, so there is a unique
real number $\rho_i$ such that
$$t_{i+1}^-=\Phi_{\rho_i}^{\xX}(t_i^+).$$
Note that $\rho_i$ has the same sign as $h'(z_i)$. Now we
define
\begin{equation}
\label{eq:chi-xX}
\chi(\xX):=\sum_{i=1}^{2r} \rho_i.
\end{equation}

\begin{lemma}
\label{lemma:chi} The number $\chi(\xX)$ only depends on $\xX$,
and not on the choices of $t_i^{\pm}$. Furthermore, endowing
the set of generic vector fields with the $\cC^{\infty}$
topology the map $\xX\mapsto \chi(\xX)$ is continuous.
\end{lemma}
\begin{proof}
We first prove that $\chi(\xX)$ does not depend on the choices
of $t_i^{\pm}$. If for any $i$ we replace $t_i^\pm$ by
$(t_i')^\pm$, then the requirement that
$(t_i')^-=\sigma_i((t_i')^+)$ implies that
$(t_i')^\pm=\Phi_{\pm\delta}^{\xX}(t_i^\pm)$ for some $\delta$,
so $\rho_i$ gets replaced by $\rho_i-\delta$ and $\rho_{i-1}$
gets replaced by $\rho_{i-1}+\delta$, and hence
(\ref{eq:chi-xX}) remains unchanged.

To prove that $\chi(\xX)$ depends continuously on $\xX$ we
first observe that any other vector field sufficiently close to
$\xX$ is also generic and has vanishing locus close to that of
$\xX$. Hence, once we have fixed the intervals $U_i$ and points
$t_i^+$ above, there is a neighborhood $\vV$ of $\xX$ in the
space of all vector fields on the circle such that if
$\yY\in\vV$ and we write $\yY=k\,\partial_x$ then
$k^{-1}(0)\subset\bigcup_i U_i$, each $U_i$ contains a unique
zero $w_i$ of $k$, and $w_i<t_i^+$. So it suffices to prove
that given $\delta>0$, choosing $\vV$ small enough, the
involution $\sigma_i^{\yY}$ satisfying (\ref{eq:prop-sigma})
with $\xX$ resp. $z_i$ replaced by $\yY$ resp. $w_i$ has the
property that $\sigma_i^{\yY}(t_i^+)$ is well defined and at
distance $<\delta$ from $\sigma_i(t_i^+)$.

The previous property will follow if we prove that $\sigma_i$
depends continuously on $\xX$. This is a local question, so let
us assume that $\xX$ is a vector field defined on an open
interval $0\in I\subset\RR$ with $\xX=g\,\partial_x$ and
satisfying $g^{-1}(0)=\{0\}$ and $g'(0)\neq 0$; by Lemma
\ref{lemma:local-models} there is an open interval $0\in
J\subset\RR$ and a smooth embedding $\phi:J\to I$ such that
$\phi_*(\lambda x\,\partial_x)=\xX$ for some nonzero real
$\lambda$.
It is easy to check that both $\phi$ and $\lambda$
depend continuously on
$\xX$.
Take $U=\phi(J\cap -J)$. The map
$\sigma:U\to U$ defined as $\sigma(x)=\phi(-\phi^{-1}(x))$
is a smooth involution of $U$ and it satisfies
$\sigma_*\xX=\xX$. By the previous observations it is clear
that $\sigma$ depends continuously on $\xX$.
\end{proof}

This is the classification theorem of nondegenerate vector fields on $S^1$:

\begin{theorem}
\label{thm:vector-fields-circle}
Given two vector fields $\xX$ and $\yY$ on the circle, there exists an orientation preserving
diffeomorphism $\phi\in\Diff^+(S^1)$ satisfying $\phi_*\xX=\yY$ if and only if
$\xX$ and $\yY$ have the same number of zeroes, the collection of derivatives at the zeroes
of $\xX$ and $\yY$, travelling along $S^1$ counterclockwise, coincide up to a cyclic permutation,
and $\chi(\xX)=\chi(\yY)$.
\end{theorem}

We are not going to prove the previous theorem. In fact we will only use
the "only if" part of it, which is rather obvious from the definitions;
the proof of the "if" part is an easy exercise using
Lemma \ref{lemma:local-models}. See \cite{By,Hi}
for detailed proofs of a more general result.

We close this subsection with another result that will be used in the proof of
Theorem \ref{thm:circle}.
Suppose that $h:\RR\to\RR$ is a smooth function such that $h(0)=h(1)=0$, and that
$h$ does not vanish on the open interval $(0,1)$. Let $\xX=h\,\partial_x$. The next
lemma follows easily from Cauchy's theorem on ODE's.

\begin{lemma}
\label{lemma:aut-interval} Any diffeomorphism $\phi:(0,1)\to
(0,1)$ satisfying $\phi_*\xX=\xX$ is equal to $\Phi_t^{\xX}$
for some $t\in\RR$. In particular if a diffeomorphism
$\phi:(0,1)\to (0,1)$ satisfying $\phi_*\xX=\xX$ is the
identity on an open subset of $(0,1)$ then $\phi$ is the
identity on the entire $(0,1)$.
\end{lemma}

\subsection{Proof of Theorem \ref{thm:circle}}
\label{ss:proof-thm:circle}

Let $\Gamma$ be a finite group acting smoothly and effectively
on $S^1$, and let $f:S^1\to\RR$ be a $\Gamma$-invariant Morse
function. Let $\Crit(f)$ be the set of critical points of $f$.
Any $\Gamma$-invariant Riemannian metric in $\mM^{\Gamma}$ is
isometric to the round circle $$\{x^2+y^2=r^2\}$$ for some
$r>0$, and this allows to identify the action of $\Gamma$ on
$S^1$ with the action of a cyclic or a dihedral group. We treat
separately the two possibilities.

\subsubsection{Dihedral groups} Suppose first that $\Gamma$ is
dihedral. Then $\Gamma$ contains elements that reverse the
orientation. Let $p\in S^1$ be a fixed point on an orientation
reversing element of $\Gamma$, and let $\Gamma_0\subset\Gamma$
be the subgroup of the elements which act preserving the
orientation. Since $[\Gamma:\Gamma_0]=2$,
we have
$\Gamma\,p=\Gamma_0\,p$. On the other hand, $p$ is
necessarily a critical point of $f$, because $f$ is
$\Gamma$-invariant.

Let $\mM_f$ be the set of metrics
$g\in\mM^{\Gamma}$ satisfying:
\begin{equation}
\label{eq:prop-mM-f-circle}
\text{if $x,y\in\Crit(f)$ and $D\nabla^gf(x)=D\nabla^gf(y)$ then
$\Gamma\,x=\Gamma\,y$.}
\end{equation}
It is clear that $\mM_f$ is open and dense in $\mM^{\Gamma}$.
Now suppose that $g\in\mM_f$ and let $\xX=\nabla^gf$. To prove
that $\Aut(\xX)/\RR\simeq\Gamma$ we consider an arbitrary
$\phi\in\Aut(\xX)$ and show that composing $\phi$ with the
action of suitably chosen elements of $\Gamma$ and with the
flow $\Phi_t^{\xX}$ for some $t$ we obtain the
identity.

Let $\phi\in\Aut(\nabla^gf)$. Composing $\phi$ with the action of some $\gamma\in\Gamma$ we
may assume that $\phi$ is orientation preserving. By
(\ref{eq:prop-mM-f-circle}) we have $\phi(p)\in\Gamma\,p$.
Since $\Gamma\,p=\Gamma_0\,p$, up to composing
$\phi$ with the action of some element of $\Gamma_0$ we can
assume that $\phi$ preserves the orientation and fixes $p$.
This implies that $\phi$ fixes all critical points of
$f$.

Let us label counterclockwise the critical points of $f$
as $p_1,p_2,\dots,p_{2r}$. By Lemma \ref{lemma:local-models},
up to composing $\phi$ with $\Phi_t^{\xX}$ for some choice of
$t$ we may assume that $\phi$ is the identity on a neighborhood
of $p_1$. This implies that $\phi$ is the identity on the
entire circle. Indeed, by Lemma \ref{lemma:aut-interval},
$\phi$ is the identity on the arc from $p_1$ to $p_2$, so by
Lemma \ref{lemma:local-models} $\phi$ is
the identity on a neighborhood of $p_2$. We next apply Lemma
\ref{lemma:aut-interval} to the arc from $p_2$ to $p_3$ and
conclude that the restriction of $\phi$ to this arc is equal to
the identity. An so on, until we have traveled around the
entire circle.

\subsubsection{Cyclic groups}
Suppose that $\Gamma$ is a cyclic group. The only case in which
$\Gamma$ can contain orientation reversing elements is that in
which $\Gamma$ consists of two elements, the nontrivial one
being an orientation reversing involution of $S^1$. This
situation can be addressed with the arguments of the previous
case, so let us assume here that all elements of $\Gamma$
preserve the orientation.
Then we define $\mM_f$ to be the set of metrics
$g\in\mM^{\Gamma}$ satisfying property (\ref{eq:prop-mM-f-circle})
above and $\chi(\nabla^gf)\neq 0$.

We claim that $\mM_f$ is open and dense in $\mM^{\Gamma}$.
Since the set of metrics $g\in\mM^{\Gamma}$ satisfying property (\ref{eq:prop-mM-f-circle})
is open and dense, to see that $\mM_f$ is dense it suffices to observe that
if for some choice of $g$ we have $\chi(\nabla^gf)=0$ then
slightly modifying $g$ away from the critical points we may
force $\chi$ to take a nonzero value; furthermore, the
modification of $g$ can be made $\Gamma$-invariant because
$\Gamma$ is generated by a rotation (note that, in contrast, if
$\Gamma$ is a dihedral group then for any $\Gamma$-invariant
metric $g$ we have $\chi(\nabla^gf)=0$).
Openness of $\mM_f$ follows from the second statement in Lemma
\ref{lemma:chi}.

Let $g\in\mM_f$, let $\xX=\nabla^gf$, and let
$\phi\in\Aut(\xX)$.
We claim that $\phi$ is orientation preserving. Indeed, for  any
orientation reversing diffeomorphism $\psi$ of $S^1$ we have
$\chi(\psi_*\xX)=-\chi(\xX)$ and since $\chi(\xX)\neq 0$, we can not possibly
have $\psi_*\xX=\xX$. Let $p$ be any critical
point of $f$. By (\ref{eq:prop-mM-f-circle})
we have $\phi(p)\in\Gamma\,p$, so up to composing $\phi$ with
the action of some element of $\Gamma$ we can assume that $\phi(p)=p$.
Then, since $\phi$ preserves the orientation, it fixes all the other
critical points, and the argument is concluded as in the case
of dihedral groups. Hence the proof of Theorem \ref{thm:circle}
is now complete.

In the remainder of the paper we are going to assume that $\dim M>1$.

\section{Equivariant Sternberg's linearisation theorem for families}
\label{s:sternberg} The following is Sternberg's
linearisation theorem \cite[Theorem 4]{St}, which extends to
the smooth setting an analytic result proved by Poincar\'e in
his thesis:

\begin{theorem}[Sternberg]
\label{thm:Sternberg} Let $0\in U\subset\RR^n$ be an open set
and let $\xX:U\to\RR^n$ be a smooth vector field satisfying
$\xX(0)=0$. Suppose that the derivative $D\xX(0)$ diagonalises
and has (possibly complex) eigenvalues
$\lambda_1,\dots,\lambda_n$, repeated with multiplicity.
Suppose that each $\lambda_i$ has negative real part, and that
\begin{equation}
\label{eq:no-resonances}
\lambda_i\neq \sum_{j=1}^n\alpha_j\lambda_j,
\qquad \text{for any }i\text{, and any }
\alpha_1,\dots,\alpha_n\in\ZZ_{\geq 0}\text{ satisfying }\sum\alpha_j\geq 2.
\end{equation}
Then there exists open sets $0\in U'\subset\RR^n$ and $0\in U''\subset U$, and a
diffeomorphism $\phi:U''\to U'$, such that $D\phi(0)=\Id$ and
$\phi\circ\xX\circ\phi^{-1}=D\xX(0)$.
\end{theorem}
Actually \cite[Theorem 4]{St} states that $\phi$ can be chosen to be
$\cC^k$ for every finite and big enough $k$. The fact
that $\phi$ can be assumed to be $\cC^{\infty}$ follows from
\cite[Theorem 6]{K}.

Sternberg proved in \cite{St} an analogous theorem for local
diffeomorphisms of $\RR^n$. Later, Anderson proved \cite[\S2,
Lemma]{A1} a parametric version of Sternberg's theorem for
diffeomorphisms, which can be translated, using the arguments
in \cite[\S 6]{St}, into a theorem on vector fields. Before
stating it, we introduce some notation. Let $D\subset\RR^n$ be
an open disk centered at $0$, and let $\Delta\subset D$ be a
smaller concentric disk. Let $r$ be a natural number. For any
smooth map $\xX:D\to\RR^n$ define
$\|\xX\|_{\Delta,r}=\sup_{x\in \Delta}\|D^r\xX(x)\|$, where
$\|D^r\xX(x)\|$ denotes the sum of the norms of all partial
derivatives of $\xX$ at $x$ of degree $\leq r$. This defines a
(non separated!) topology on $\Map_0(D,\RR^n)$, the set of all
smooth maps $D\to\RR^n$ fixing $0$, and we denote by
$\Map_0(D,\RR^n)_{\Delta,r}$ the resulting topological space.
This is the analogue of Anderson's theorem for vector fields:

\begin{theorem}
\label{thm:Anderson} Let $L:\RR^n\to\RR^n$ be a linear map
which diagonalises with eigenvalues $\lambda_1,\dots,\lambda_n$
satisfying (\ref{eq:no-resonances}). Assume that each
$\lambda_i$ has negative real part, and that
$\lambda_i\neq\lambda_j$ for $i\neq j$.
There exists a neighborhood $N$ of $L$ in
$\Map_0(D,\RR^n)_{\Delta,r+1}$ and a continuous map
$$\Phi:N\to\Map_0(D,\RR^n)_{\Delta,r}$$ such that:
\begin{enumerate}
\item for every $\xX\in N$, $D\Phi(\xX)(0)=\Id$, so
    $\Phi(\xX)$ gives a diffeomorphism $U_{\xX}\to
    U'_{\xX}$ between neighborhoods of $0$,
\item for every $\xX\in N$,
    $\Phi(\xX)\circ\xX\circ\Phi(\xX)^{-1}:U'_{\xX}\to
    \RR^n$ is equal to $D\xX(0)$.
\end{enumerate}
\end{theorem}

We will need an analogue of Theorem
\ref{thm:Anderson} in an equivariant setting. However, as was
mentioned in the introduction, the presence of symmetries
usually forces eigenvalues to have high multiplicity, and
consequently the hypothesis in Theorem \ref{thm:Anderson} will
most of the times not hold.

Now the (only) reason why Anderson assumes the eigenvalues
$\lambda_1,\dots,\lambda_n$ to be pairwise distinct is that he
needs to be able to diagonalize the linear maps close to $L$ in
a continuous way. To state this more precisely, let
$\GL^*(\RR,n)\subset\GL(n,\RR)$ denote the open and dense set
of linear automorphisms of $\RR^n$ all of whose eigenvalues are
distinct. Anderson uses the following elementary lemma.

\begin{lemma}
\label{lemma:diag}
Any $L\in\GL^*(n,\RR)$ admits a neighborhood $U\subset\GL^*(n,\RR)$
and smooth maps $f_1,\dots,f_n:U\to\CC^n$ so that for any
$L'\in U$ the vectors $f_1(L'),\dots,f_n(L')$ form a basis of
$\CC^n$ with respect to which $L'$ diagonalizes.
\end{lemma}

So to obtain an equivariant analogue of Theorem
\ref{thm:Anderson} it suffices to define some open and dense
subset of the set of equivariant automorphisms of a vector
space enjoying the same property as $\GL^*(n,\RR)$. This is the purpose
of the following lemma, which also proves a property on centralizers
that will be used later in the paper.

Suppose that $V$ is an $n$-dimensional real vector space, and
that a finite group $G$ acts linearly on $V$. Denote the
centralizer of any $\Lambda\in\Aut(V)$ by
$$Z(\Lambda)=\{\Lambda'\in\Aut(V)\mid \Lambda\Lambda'=\Lambda'\Lambda\}.$$
Let $\Aut_G(V)$ denote the Lie group of automorphisms of $V$
commuting with the $G$ action.
Define $\Aut^*_G(V)$ to be the set of all $\Lambda\in\Aut_G(V)$
such that for any $\lambda\in\CC$ the ($G$-invariant) subspace
$\Ker(\Lambda-\lambda\Id)\subset V$ is irreducible as a
representation of $G$. Given a basis $a_1,\dots,a_n\in V\otimes\CC$ we denote by
$(a_1,\dots,a_n):\CC^n\to V\otimes\CC$ the isomorphism
$(\lambda_1,\dots,\lambda_n)\mapsto\sum\lambda_ia_i$.

\begin{lemma}
\label{lemma:automorfismes-Gamma-generics} The subset
$\Aut^*_G(V)$ is open and dense in $\Aut_G(V)$. Any
$\Lambda\in\Aut^*_G(V)$ has a neighborhood $U\subset
\Aut^*_G(V)$ with smooth maps $f_1,\dots,f_n:U\to V\otimes\CC$
so that for any $\Lambda'\in U$ the vectors
$f_1(\Lambda'),\dots,f_n(\Lambda')$ form a basis of
$V\otimes\CC$ with respect to which $\Lambda'$ diagonalizes, and
conjugation by $(f_1',\dots,f_n')(f_1,\dots,f_n)^{-1}$
gives an isomorphism
$$Z(\Lambda)\stackrel{\simeq}{\longrightarrow}Z(\Lambda').$$
\end{lemma}
\begin{proof}
Denote by $\wh{G}$ the set of irreducible characters of $G$.
For any $\xi\in \wh{G}$ let $V_{\xi}$ be a $G$-representation
with character $\xi$. As a
$G$-representation, we may identify $V$ with $\bigoplus_{\xi\in\wh{G}}V_{\xi}\otimes
E_{\xi}$, where each $E_{\xi}$ is a vector space with trivial
$G$-action. By Schur's lemma the space of $G$-equivariant
endomorphisms of $V$ is
$$\End_G(V)=\bigoplus_{\xi\in\wh{G}}\End E_{\xi}.$$
An endomorphism $\Lambda=(\Lambda_{\xi})_{\xi}$ (where
$\Lambda_{\xi}\in\End E_{\xi}$ for each $\xi$) belongs to
$\Aut_G(V)$ exactly when $\prod_{\xi}\det\Lambda_{\xi}\neq 0$,
and it belongs to $\Aut_G^*(V)$ if and only if, additionally,
no root of the polynomial
$\prod_{\xi}\det(\Lambda_{\xi}-x\Id_{E_{\xi}})\in\RR[x]$ has
multiplicity bigger than one. This condition implies that
$\Lambda_{\xi}\in\GL^*(E_{\xi})$ for each $\xi$. Applying Lemma \ref{lemma:diag}
to each $\Lambda_{\xi}$ we deduce the existence
of a neighborhood $U\subset
\Aut^*_G(V)$ of $\Lambda$ and smooth maps $f_1,\dots,f_n:U\to V\otimes\CC$
and $\lambda_1,\dots,\lambda_n:U\to\CC$
so that for any $\Lambda'\in U$ we have $\Lambda'(f_j(\Lambda'))=\lambda_j(\Lambda') f_j(\Lambda')$
for every $j$.
For any $\Lambda'\in U$ we can identify
$Z(\Lambda')$ with the subgroup of $\Aut(V)$ preserving the subspace of $V\otimes\CC$
spanned by $\{f_j(\Lambda')\mid\lambda_j(\Lambda')=\lambda\}$ for each $\lambda$.
Shrinking $U$ if necessary we may assume that for any $i,j$ and any $\Lambda'\in U$ we have
$$\lambda_i(\Lambda')=\lambda_j(\Lambda')\quad\Longleftrightarrow\quad
\lambda_i(\Lambda)=\lambda_j(\Lambda),$$
so conjugation by $(f_1',\dots,f_n')(f_1,\dots,f_n)^{-1}$
gives an isomorphism
$Z(\Lambda)\stackrel{\simeq}{\longrightarrow}Z(\Lambda')$.
\end{proof}

Take some $G$-invariant Euclidean metric on $V$, let $D\subset
V$ be an open disk centered at $0$, and let $\Delta\subset D$
be a smaller concentric disk. Let $r$ be a natural number. For
any smooth map $\xX:D\to V$ define
$\|\xX\|_{\Delta,r}=\sup_{x\in \Delta}\|D^r\xX(x)\|$ as before.
This defines a topology on $\Map_{G,0}(D,V)$, the set of all
$G$-equivariant smooth maps $D\to V$ fixing $0$. Let
$\Map_{G,0}(D,V)_{\Delta,r}$ be the resulting topological
space. Define analogously $\Map_{0}(D,V)_{\Delta,r}$ by
dropping the equivariance condition. Combining the previous
lemma with the arguments in \cite[\S2, Lemma]{A1} and \cite[\S
6]{St} we obtain the following.

\begin{theorem}
\label{thm:Anderson-equivariant} Let $L\in\Aut_G^*(V)$ have
eigenvalues $\lambda_1,\dots,\lambda_n$ satisfying
(\ref{eq:no-resonances}) and suppose that each $\lambda_i$ has
negative real part. There is a neighborhood $N$ of $L$ in
$\Map_{G,0}(D,V)_{\Delta,r+1}$ and a continuous map
$$\Phi:N\to\Map_{G,0}(D,V)_{\Delta,r}$$ such that:
\begin{enumerate}
\item for every $\xX\in N$, $D\Phi(\xX)(0)=\Id$, so
    $\Phi(\xX)$ gives a diffeomorphism $U_{\xX}\to
    U'_{\xX}$ between neighborhoods of $0$,
\item for every $\xX\in N$,
    $\Phi(\xX)\circ\xX\circ\Phi(\xX)^{-1}:U'_{\xX}\to V$ is
    equal to $D\xX(0)$.
\end{enumerate}
\end{theorem}

The only part in the statement of Theorem \ref{thm:Anderson-equivariant}
that does not follow immediately is the fact
that the conjugating map $\Phi$ may be chosen to take values in
$\Map_{G,0}(D,V)_{\Delta,r}$. Sternberg's
argument provides a (continuous, by Anderson) map
$\Phi_0:N\to\Map_{0}(D,V)_{\Delta,r}$, satisfying (i)
$D\Phi_0(\xX)(0)=\Id$ and (ii)
$\Phi_0(\xX)\circ\xX\circ\Phi_0(\xX)^{-1}=D\xX(0)$ (in a
neighborhood of $0$), but $\Phi_0(\xX)$ is not necessarily
equivariant. Now, equality (2) is equivalent to
\begin{equation}
\label{eq:conjugacio-Phi-0}
\Phi_0(\xX)\circ\xX=D\xX(0)\circ \Phi_0(\xX),
\end{equation}
so setting
$$\Phi(\xX)(x)=\frac{1}{|G|}\sum_{g\in G}g\Phi_0(\xX)(g^{-1}x)\in V$$
for every $x\in D$, we have $\Phi(\xX)\in\Map_{G,0}(D,V)_{\Delta,r}$,
and equation (\ref{eq:conjugacio-Phi-0})
immediately gives $\Phi(\xX)\circ\xX=D\xX(0)\circ \Phi(\xX)$.
Trivially we also have $D\Phi(\xX)(0)=\Id$ for every $\xX$, and
$\Phi(\xX):D\to V$ is $G$-equivariant. The map
$\Phi:N\to\Map_{G,0}(D,V)_{\Delta,r}$ is continuous, because
$\Phi_0$ is, so now Theorem \ref{thm:Anderson-equivariant} is
clear.

\section{The space of metrics $\mM_0$}
\label{s:mM-0-mM-1}

\subsection{Preliminaries}

\newcommand{\codim}{\operatorname{codim}}
\newcommand{\free}{\operatorname{free}}

The following result is a standard consequence of the existence of linear
slices for smooth compact group actions (see e.g. \cite[Chap. VI, \S 2]{Br}).

\begin{lemma}
\label{lemma:linearisation}
Let $G$ be a finite group acting smoothly on a connected manifold $X$.
\begin{enumerate}
\item For each subgroup $H\subseteq G$ the fixed point set
    $X^H=\{x\in X\mid H\subseteq G_x\}$ is the disjoint union of finitely many closed
    submanifolds of $X$ (not necessarily of the same dimension) satisfying
    $T_x(X^H)=(T_xX)^H$ for every $x\in X^H$. In
    particular, either $X^H=X$ or $X^H$ has empty interior.
\item Assume that the action of $G$ on $X$ is effective.
    Then $X^{\free}=\{x\in X\mid G_x=\{1\}\}$ is open and
    dense in $X$.
\end{enumerate}
\end{lemma}

%
%

\subsection{Sinks, sources, and (un)stable manifolds}
\label{ss:sink,sources}
Let $n>1$ and let $M$ be a compact connected
$n$-dimensional manifold. Suppose that $M$ is
endowed with a smooth and effective action of a finite group
$\Gamma$. Denote the stabilizer of any $x\in M$ by
$$\Gamma_x=\{\gamma\in\Gamma\mid\gamma x=x\}.$$
Let $\mM$ denote the space of Riemannian metrics on $M$, and let $\mM^{\Gamma}\subset\mM$ be the
subset of $\Gamma$-invariant metrics.

Let
$$f:M\to\RR$$
be a $\Gamma$-invariant Morse function. This function will be
fixed throughout the rest of the paper.
If $p$ is a critical point of $f$, so that $\nabla^gf(p)=0$,
the derivative $D\nabla^gf(p)$ is a well defined endomorphism
of $T_pM$ (one may define it using a connection on $TM$, but
the result will be independent of the chosen connection). The endomorphism
$D\nabla^gf(p)$ is self adjoint
with respect to the Euclidean norm on $T_pM$ given by $g$, so
$D\nabla^gf(p)$ diagonalizes.

Denote the index of a critical
point $p$ of $f$ by $\Ind_f(p)$.
Let $\Crit(f)\subset M$ be the set of critical points of $f$,
and for any $k$ let
$$\Crit_k(f)=\{p\in\Crit(f)\mid \Ind_f(p)=k\}.$$
Define the set of sinks
of $f$ to be $\II=\Crit_n(f)$ and the set of sources to be
$\OO=\Crit_0(f)$. The
points in $\II$ (resp. $\OO$) are the sinks (resp. sources)
of the the gradient vector field $\nabla^gf$ for every $g$.
Denote also by $\EE=\II\cup\OO$ the collection of all local
extremes of $f$.

For any $g\in\mM$ and any real number $t$ let $\Phi^g_t:M\to M$
denote the flow at time $t$ of $\nabla^gf$. Define the stable
and unstable manifolds of $p\in\Crit(f)$ to be, respectively,
$$W^s_g(p)=\{q\in M\mid\lim_{t\to\infty}\Phi^g_t(q)=p\},
\qquad W^u_g(p)=\{q\in M\mid\lim_{t\to-\infty}\Phi^g_t(q)=p\}.$$

For any $p\in\EE$ and any $g\in\mM^{\Gamma}$ let
$$L_g(p)=\{\psi\in \Aut(T_pM)\mid (D\nabla^gf(p)) \psi=\psi(D\nabla^gf(p))\}=Z(D\nabla^gf(p)).$$
Since $\Gamma$ is finite and acts effectively on $M$, we can
identify $\Gamma_p$ with a subgroup of $L_g(p)$ using (1) in Lemma \ref{lemma:linearisation} above.

\subsection{The metrics in $\mM_0$: generic eigenvalues at critical points}
\label{ss:mM-0}
Let $$\mM_0\subset\mM^{\Gamma}$$ denote set of $\Gamma$-invariant
metrics $g$ satisfying the following conditions:
\begin{enumerate}
\item[(C1)] for any $p\in\EE$ the eigenvalues $\lambda_1,\dots,\lambda_n$
of the linearization $D\nabla^gf(p)$ satisfy condition (\ref{eq:no-resonances}) in
Theorem \ref{thm:Sternberg};
\item[(C2)] if $p,q\in\EE$, then the eigenvalues of
    $D\nabla^gf(p)$ and $D\nabla^gf(q)$ coincide if and
    only if $p$ and $q$ belong to the same $\Gamma$-orbit;
\item[(C3)] for any $p\in\EE$ we have  $D\nabla^gf(p)\in\Aut_{\Gamma_p}^*(T_pM)$.
\end{enumerate}

Condition (C1), combined with Sternberg's Theorem \ref{thm:Sternberg} and
an easy adaptation of a theorem of Kopell \cite[Theorem 6]{K}
from maps to vector fields, implies that if $p\in\II$ then
the map
\begin{equation}
\label{eq:Aut-L}
D(p):\Aut(\nabla^gf|_{W^s_g(p)})\to L_g(p)
\end{equation}
sending any $\phi\in \Aut(\nabla^gf|_{W^s_g(p)})$ to
$D\phi(p)\in L_g(p)$ is an isomorphism (it is clear that any
such $\phi$ fixes $p$); furthermore, there is a diffeomorphism
$h(p):T_pM \to W^s_g(p)$ making the following diagram
commutative:
\begin{equation}
\label{eq:accio-linealitzada}
\xymatrix{L_g(p)\times T_pM \ar[d]_{D(p)^{-1}\times h(p)}\ar[r] & T_pM \ar[d]^{h(p)} \\
\Aut(\nabla^gf|_{W^s_g(p)})\times W^s_g(p)\ar[r] & W^s_g(p),}
\end{equation}
where the horizontal arrows are the maps defining the actions.

\begin{remark}
\label{rmk:primer-entorn}
Strictly speaking, Sternberg's theorem gives a diffeomorphism
between a neigborhood of $0$ in $T_pM$ and a neighborhood of
$p$ in $W^s_g(p)$ which commutes the flows of $D\nabla^gf(p)$
and of $\nabla^gf$, but such diffeomorphism can be extended
uniquely imposing compatibility with the flows to yield $h(p)$.
\end{remark}

Similarly, for any source $p\in\OO$ the analogous map
$\Aut(\nabla^gpf|_{W^u_g(p)})\to L_g(p)$ is an isomorphism and
there is a diffeomorphism $T_pM \to W^u_g(p)$ which is
equivariant in the obvious sense, analogous to the case of
sinks.

Condition (C2) implies that for any $\phi\in\Aut(\nabla^gf)$
and any $p\in\EE$ we have $\phi(p)=\gamma p$ for some
$\gamma\in\Gamma$. Of course a priori $\gamma$ may depend on
$p$, but in the course of proving Theorem \ref{thm:main} we
will deduce that for $g$ belonging to a residual subset of
$\mM_0$ and any $\phi\in\Aut(\nabla^gf)$, there exists some
$\gamma$ such that $\phi(p)=\gamma p$ for each $p\in\EE$.

By Lemma \ref{lemma:automorfismes-Gamma-generics} $\mM_0$ is
open and dense in $\mM^{\Gamma}$. Moreover, combining (C3)
with Lemma
\ref{lemma:automorfismes-Gamma-generics} and Theorem
\ref{thm:Anderson-equivariant} (together with the obvious analogue
of Remark \ref{rmk:primer-entorn}) we deduce the following
result.

\begin{lemma}
\label{lemma:C1-families} Any $g\in\mM_0$ has a neighborhood
$\uU\subset\mM_0$ such that for any $p\in\EE$ the following holds.
Let $V_p=T_pM$. Endow the space of maps $\Map(V_p,M)$
with the weak (compact-open) $\cC^{\infty}$-topology \cite[Chap 2, \S1]{H}.
For any $g'\in\uU$ there is a linear vector field
$$\xX_{g'}(p):V_p\to V_p$$
depending continuously on $g'$ and a $\Gamma_p$-equivariant embedding
$$h_{g'}(p):V_p\to M$$
depending also continuously on $g'$ with the following properties.
\begin{enumerate}
\item $Z(\xX_{g'}(p))=Z(\xX_g(p))=L_g(p)$ for every
    $g'\in\uU$,
\item $h_{g'}(p)$ identifies $\xX_{g'}(p)$ with the restriction of $\nabla^{g'}f$ to $h_{g'}(p)(V_p)$;
hence, $$h_{g'}(p)(V_p)=W_g^s(p)\qquad\text{ if $p\in\II$}$$ and
$$h_{g'}(p)(V_p)=W_g^u(p)\qquad\text{ if $p\in\OO$}.$$
\end{enumerate}
\end{lemma}

Note that we do not claim that the derivative of $h_{g'}(p)$ at $p$ is the identity: in fact
in general this will not be the case (otherwise we could not pretend to have the identifications
$Z(\xX_{g'}(p))=Z(\xX_g(p))$).

\section{The spheres $S_g(p)$, the distributions $\aA_g(p)$, and the sets $F_g(p)$}

Recall that we assume $\dim M>1$.

\subsection{The spheres $S_g(p)$}
\label{ss:spheres}
For any $g\in\mM$ and any sink $p\in\II$ we denote by $\sim$
the equivalence relation in $W^s_g(p)$ that identifies two
points whenever they belong to the same integral curve of
$\nabla^gf$. We then define
$$S_g(p)=(W^s_g(p)\setminus\{p\})/\sim.$$
Let $\epsilon>0$ be a real number and let $\Sigma\subset M$ be
the $g$-geodesic sphere of radius $\epsilon>0$ and center $p$.
If $\epsilon$ is small enough (which we assume), then $\Sigma$
is a submanifold of $M$ diffeomorphic to $S^{n-1}$ and every
equivalence class in $S_g(p)$ contains a unique representative
in $\Sigma$. Hence, composing the inclusion
$\Sigma\hookrightarrow W^s_g(p)\setminus\{p\}$ with the
projection
$$\pi_p:W^s_g(p)\setminus\{p\}\to S_g(p)$$
gives a bijection $\Sigma\simeq S_g(q)$. This allows us to
transport the smooth structure on $\Sigma$ to a smooth
structure (in particular, a topology) on $S_g(p)$, independent
of $\epsilon$.

If $g\in\mM_0$ then the action of $L_g(p)$ on $S_g(p)$ defined
via the identification (\ref{eq:Aut-L}) is smooth, and so is
the natural action of $\Gamma_p$ on $S_g(p)$ (recall that
$\mM_0\subset\mM^{\Gamma}$).

For any $p\in\II$ and $q\in\OO$ let
$$\Omega_g(p,q):=\pi_p(W^s_g(p)\cap W^u_g(q))\subset S_g(p).$$
Since $W_g^u(q)$ is open in $M$, $\Omega_g(p,q)$ is open in $S_g(p)$.

Similarly, if $q$ is a source we define
$$S_g(q)=(W_g^u(q)\setminus\{q\})/\sim$$
and we denote by
$$\pi_q:W_g^u(q)\setminus\{q\}\to S_g(q)$$
the projection. If $p$ is a sink, then we define
$$\Omega_g(q,p)=\pi_q(W^s_g(p)\cap W^u_g(q)),$$
which is an open subset of $S_g(q)$.

For convenience, if
$p,q\in\II$ or $p,q\in\OO$ we define $\Omega_g(p,q)=\emptyset$.

Since the fibers of the restrictions of $\pi_p$ and $\pi_q$ on
$W^s_g(p)\cap W^u_g(q)$ are the same, there are natural
bijections
$$\sigma_g^{p,q}:\Omega_g(p,q)\to\Omega_g(q,p),\qquad\sigma_g^{q,p}=(\sigma_g^{p,q})^{-1},$$
which are easily seen to be diffeomorphisms.

\subsection{The singular distributions $\aA_g(q)$}
\label{ss:singular-distributions} Assume through the remainder
of this section that $g\in\mM_0$. We will consider, for every
$p\in\EE$, the diagonal action of $L_g(p)$ on $S_g(p)^k$ for
some natural number $k$. If $z=(z_1,\dots,z_k)\in S_g(p)^k$ and
$\psi\in L_g(p)$ we denote
$$\psi z=(\psi z_1,\dots,\psi z_k).$$ Similarly, we will
consider the diagonal extension of the maps $\sigma_g^{p,q}$:
$$\sigma_g^{p,q}:\Omega_g(p,q)^k\to \Omega_g(q,p)^k,\qquad
\sigma_g^{p,q} z=(\sigma_g^{p,q} z_1,\dots,\sigma_g^{p,q} z_k).$$
We are going to use below without explicit notice analogous
diagonal extensions of maps to Cartesian products.

Let $n=\dim M$. Define
\begin{equation}
\label{eq:def-r}
r:=\left[\frac{2n^2}{n-1}+1\right].
\end{equation}
The choice of this number will be justified in the proof of
Lemma \ref{lemma:mM-2-K-dense}.

For any $q\in\EE$ we denote by $\aA_g(q)\subset T(S_g(q)^r)$
the subspace consisting of all tangent vectors given by the
infinitesimal action of the Lie algebra of $L_g(q)$. This
gives, for any $z\in S_g(q)^r$, a linear subspace
$\aA_g(q)(z)\subset T_xS_g(q)^r$ whose dimension may vary with
$z$ (hence, one can think of $\aA_g(q)$ as a singular
distribution). In concrete terms,
$$\aA_g(q)(z)=\{\yY_{g,\sigma}(z)\mid \sigma\in\Lie L_g(q)\},$$
where for any $\sigma\in\Lie L_g(q)$ we denote by $\yY_{g,\sigma}$ the vector field
on $S_g(q)^r$ given by the infinitesimal action of $\sigma$.

\subsection{The subset $F_g(q)\subset S_g(q)^r$}
\label{ss:F-g-q}
We next want to identify a dense open subset of $S_g(q)^r$ on
which the action of $L_g(q)$ has the smallest possible isotropy
subgroup, and on which $\aA_g(q)$ restricts to a vector
subbundle of $T(S_g(q)^r)$. We remark that, since $L_g(q)$ is an infinite group,
in this situation we can not use (2) in Lemma \ref{lemma:linearisation}.
Let
$$\xX_g(q):=D\nabla^gf(p)\in \Lie L_g(q).$$ Note that
$e^{t\xX_g(q)}$ corresponds, via the isomorphism $D(q)$ in
(\ref{eq:Aut-L}), to the flow $\Phi^g_t$, so $e^{t\xX_g(p)}$
acts trivially on $S_g(q)$ and hence on $S_g(q)^r$.

Let us denote $V=T_qM$. Then $X:=\xX_g(q)$ is a diagonalizable endomorphism of $V$.
Denote its eigenvalues by $\lambda_1,\dots,\lambda_k$. Let
$V_j\subseteq V$ be the subspace consisting of eigenvectors with eigenvalue $\lambda_j$.
We have a decomposition $V=V_1\oplus\dots\oplus V_k$ with respect to which we may define
projections $\pi_j:V\to V_j$. Let us say that a collection of vectors $w_1,\dots,w_s\in V_j$
is thick if $s>d_j=\dim V_j$ and for any $1\leq i_1<i_2<\dots<i_{d_j}\leq s$ the vectors
$w_{i_1},\dots,w_{i_{d_j}}$ are linearly independent. Finally, we say that a collection of
vectors $v_1,\dots,v_s\in V$ is thick if for any $j$ the projections
$\pi_j(v_1),\dots,\pi_j(v_s)$ form a thick collection of vectors in $V_j$.
Let $G=L_g(q)$.

\begin{lemma}
\label{lemma:thick-free}
Suppose that $v_1,\dots,v_s$ is a thick collection of vectors, and that for
some $g\in G$ there exist real numbers $t_1,\dots,t_s$ satisfying
$gv_j=e^{t_jX}v_j$ for every $j$. Then $g=e^{tX}$ for some real number $t$.
\end{lemma}
\begin{proof}
Consider first the case $k=1$, so that $X$ is a homothecy. Write
$v_{n+1}=a_1v_1+\dots+a_nv_n$. The thickness condition implies that
$a_i\neq 0$ for every $i$. By assumption we have $gv_i=\lambda_iv_i$ for
some real numbers $\lambda_1,\dots,\lambda_s$. In particular,
$$\lambda_{n+1}(a_1v_1+\dots+a_nv_n)=\lambda_1 a_1v_1+\dots+\lambda_n a_nv_n.$$
Taking into account that $v_1,\dots,v_n$ is a basis and equating coefficients we deduce that
$\lambda_{n+1}=\lambda_1=\dots=\lambda_n$. So the case $k=1$ is proved.
The case $k>1$ follows from applying the previous arguments to each $V_j$,
using the fact that every $g\in G$ preserves $V_j$.
\end{proof}

Let $S(V)$ denote the set of orbits of $H=\{e^{tX}\mid t\in\RR\}$ acting on $V\setminus\{0\}$.
$H$ is a central subgroup of $G$, and the action of $G$ on $V$ induces an action of $G/H$ on $S(V)$.
Let $F\subset S(V)^r$ denote the set of tuples $(x_1,\dots,x_r)$ such that, writing
$x_i=Hx_i'$ with $x_i'\in V$ for each $i$, the vectors $x_1',\dots,x_r'$ form a thick collection
(this is independent of the choice of representatives $x_i'$).

\begin{lemma}
\label{lemma:set-F}
\begin{enumerate}
\item $F$ is a dense an open subset of $S(V)^r$;
\item the restricted action of $G/H$ on $F$ is free.
\end{enumerate}
\end{lemma}
\begin{proof}
For (1) note that $r>n$, so the set $F'$ of thick $r$-tuples in $V^r$
can be identified with the complementary of finitely many proper subvarieties
(those corresponding to the possible linear relations among projections to each
summand $V_j$ of subsets of the tuple, given by the vanishing of suitable
determinants). Hence $F'$ is a dense open subset of
$V^r$, which implies that $F\subset S(V)^r$ is open and dense. (2) follows from
Lemma \ref{lemma:thick-free}.
\end{proof}

Assume that $q$ is a sink.
Choose a diffeomorphism $h:V\to W_g^s(q)$ making commutative
the diagram (\ref{eq:accio-linealitzada}) with $p$ replaced by $q$. Then
$h$ induces a diffeomorphism $S(V)\to S_g(q)$, which can be extended linearly
to $S(V)^r\to S_g(q)^r$. Let
$$F_g(q)\subset S_g(q)^r$$
be the image of $F$ under the previous diffeomorphism. The set $F_g(q)$
is independent of the choice of $h$. Indeed, two different choices of
$h$ differ by precomposition with an element of $G$, and the action of
$G$ on $S(V)^r$ preserves $F$. If instead $q$ is a source, consider the
same definition with $W_g^s(q)$ replaced by $W_g^u(q)$.

Lemma \ref{lemma:set-F} and an obvious estimate
imply:

\begin{lemma}
\label{lemma:restriction-aA}
\begin{enumerate}
\item If $z\in F_g(q)$ and $\psi\in L_g(q)$ satisfies $\psi
    z=z$ then $\psi =e^{t\xX_g(q)}$ for some $t\in\RR$.
\item The restriction of $\aA_g(q)$ to $F$ is a vector bundle
    of rank $\dim G-1\leq n^2-1$.
\end{enumerate}
\end{lemma}

\section{The space of metrics $\mM_{1,K}$}
\label{s:mM-2-K}

We recall again that $\dim M>1$.

\subsection{Definition of $\mM_{1,K}$}
\label{ss:def-mM-2} Let $g\in\mM_0$. Let $p\in\EE$ and let $K$
be a natural number. Denote by $\|\cdot\|_g$ the operator norm
in $\End T_pM$ induced by $g$. Denote by $$L_{g,K}(p)\subset
L_g(p)$$ the subset consisting of those $\psi\in L_g(p)$ such
that $\|\psi\|_g\leq K$, $\|\psi^{-1}\|_g\leq K$, and
$$\|\psi-e^{t\xX_g(p)}\gamma\|_g\geq K^{-1} \text{ for every
$t\in\RR$ and $\gamma\in\Gamma_p$}.$$ Clearly $L_{g,K}(p)$ is compact.
Recall that the number
$r$ has been defined in (\ref{eq:def-r}) in Subsection
(\ref{ss:singular-distributions}) above. For any $\psi\in
L_g(p)$ we denote by
$$\alpha_\psi:S_g(q)^r\to S_g(q)^r$$
the map given by the action of $\psi$.

\begin{definition}
\label{def:mM_{1,K}} Let $p\in\EE$. Define $\mM_{1,K}(p)$ as
the set of all metrics $g\in\mM_0$ such that for any $\psi\in
L_{g,K}(p)$ there exist:
\begin{enumerate}
\item $q,q'\in\EE$ and $z\in\Omega_g(p,q)^r$ satisfying
    $$\psi z\in\Omega_g(p,q')^r,\qquad \sigma_g^{p,q}z\in F_g(q),
    \qquad \sigma_g^{p,q'}\psi z\in F_g(q'),$$
\item and a vector $u\in\aA_g(q)(\sigma_g^{p,q}z)\subset
    T_{\sigma_g^{p,q}z}S_g(q)^r$ such that
\begin{equation}
\label{eq:xi-no-funciona}
D(\sigma_g^{p,q'}\circ \alpha_\psi\circ \sigma_g^{q,p})(u)\notin
\aA_g(q')(\sigma_g^{p,q'}\circ \alpha_\psi(z)).
\end{equation}
\end{enumerate}
\end{definition}

Here $D(\sigma_g^{p,q'}\circ \alpha_\psi\circ \sigma_g^{q,p})$
is the map between tangent spaces given by the differential of
$$\sigma_g^{p,q'}\circ \alpha_\psi\circ \sigma_g^{q,p}:
\sigma_g^{p,q}(\alpha_{\psi}^{-1}(\Omega_g(p,q')^r))\to \Omega_g(q',p)^r,$$
and $\sigma_g^{p,q}(\alpha_{\psi}^{-1}(\Omega_g(p,q')^r))$ is
an open subset of $\Omega_g(q,p)$ containing $\sigma_g^{p,q}z$.

Define finally:
$$\mM_{1,K}=\bigcap_{p\in\EE}\mM_{1,K}(p).$$

\subsection{$\mM_{1,K}$ is open and dense in $\mM_0$}

\begin{lemma}
\label{lemma:mM-2-K-open}
$\mM_{1,K}(p)$ is an open subset of $\mM_0$.
\end{lemma}
\begin{proof}
Let $g\in\mM_{1,K}(p)$. If $\psi\in L_{g,K}(p)\subset L_g(p)$
and $(q,q',z,u)$ satisfy (1) and (2) in Definition
\ref{def:mM_{1,K}}, then we say that $(q,q',z,u)$ {\it rules
out} $\psi$. The set of elements in $L_g(p)$ which are ruled
out by any given tuple $(q,q',z,u)$ is open. Since $L_{g,K}(p)$
is compact, it follows that there exist finitely many tuples
$(q_1,q'_1,z_1,u_1),\dots,(q_{\nu},q'_{\nu},z_{\nu},u_{\nu})$
and open subsets $V_1,\dots,V_{\nu}\subset L_g(p)$ such that
$L_{g,K}(p)\subset V_1\cup\dots\cup V_{\nu}$ and such that, for
every $j$, $(q_j,q'_j,z_j,u_j)$ rules out each element of
$V_j$. Choose subsets $V_j'\subset V_j$ with the property that
$L_{g,K}(p)\subset V_1'\cup\dots\cup V_{\nu}'$, and such that
$\ov{V_j'}$ is compact and contained in $V_j$ for each $j$.

Applying Lemma \ref{lemma:C1-families} to $g$ we deduce the
existence of a neighborhood $\uU\subset\mM_0$ of $g$ and
natural smooth identifications $S_g(q)\simeq S_{g'}(q)$ for
every $g'\in\uU$ and $q\in\EE$. Since in the remainder of the
proof we only consider metrics from $\uU$, we denote $S(q)$
instead of $S_{g'}(q)$. We also get for every $g'\in\uU$
natural isomorphisms of groups $L_g(q)\simeq L_{g'}(q)$ which
are compatible with both inclusions of $\Gamma_z$ in $L_g(q)$
and $L_{g'}(q)$ and with the identifications $S_{g}(q)\simeq
S_{g'}(q)$, and for this reason we write $L(q)$ instead of
$L_{g'}(q)$. Now we may view $V_1,\dots,V_{\nu}$ as subsets of
$L(p)$. Shrinking $\uU$ if necessary we may assume that
$L_{g',K}(p)\subset V_1'\cup\dots\cup V_{\nu}'$ for every
$g'\in\uU$.

The sets $F_{g'}(q)\subset S(q)^r$ are independent of $g'$ and
the distributions $\aA_{g'}(q)$ vary continuously with $g'$.
Similarly the subsets $\Omega_{g'}(p,q)\subset S(p)$ vary
continuously with $g'$, meaning that, for any
$z\in\Omega_g(p,q)$, if $g'$ is sufficiently close to $g$ then
$z\in\Omega_{g'}(p,q)$ as well. The maps
$\sigma_{g'}^{p,q}:\Omega_{g'}(p,q)\to\Omega_{g'}(q,p)$ also
depend continuously on $g'$ in the obvious sense.

For any $g'\in\uU$, any $q\in\EE$, and any $z\in F(q)$ we
define $P_{g'}:T_zS(q)^r\to\aA_{g'}(q)(z)$ to be the orthogonal
projection with respect to $g'$ (recall that $\aA_{g'}(q)(z)$
is a vector subspace of $T_zS(q)^r$).
The previous observations imply the following: for every $1\leq
j\leq \nu$ there exists a neighborhood of $g$,
$\uU_j\subset\uU$, such that for every $g'\in \uU_j$ and any
$\psi'\in \overline{V_j'}$, the tuple
$(q_j,q_j',z_j,P_{g'}(u_j))$ rules out $\psi'$. It then follows
that $\uU_1\cap\dots\cap\uU_{\nu}\subset\mM_{1,K}(p)$, and
hence $\mM_{1,K}(p)$ is open.
\end{proof}

\begin{lemma}
\label{lemma:mM-2-K-dense}
$\mM_{1,K}(p)$ is a dense subset of $\mM_0$.
\end{lemma}
\begin{proof}
We will use the following lemma, whose proof
is postponed to the Appendix.

\begin{lemma}
\label{lemma:variacions-metrica-diff} Suppose that
$g\in\mM^{\Gamma}$, $x\in M\setminus\Crit(f)$ and
$y=\Phi^g_t(x)$ for some nonzero $t$. Suppose that the
stabilizer $\Gamma_x$ is trivial. Let $v\in T_xM$ be a nonzero
vector, and let $w=D\Phi^g_t(x)(v)$. Given any $u\in T_v(TM)$
and any $\Gamma$-invariant open subset $U\subset M$ containing
$\Phi^g_{t'}(x)$ for some $t'\in (0,t)$, there exists some
$g'\in\cC^{\infty}(M,S^2T^*M)^{\Gamma}$ supported on $U$ such
that
$$\left.\frac{\partial}{\partial\epsilon}D\Phi^{g+\epsilon g'}_{-t}(w)\right|_{\epsilon=0}=u.$$
\end{lemma}

Fix some $g\in\mM_0$. We assume for concreteness throughout the
proof that $p$ is a sink. The case in which $p$ is a source
follows from the same arguments (or replacing $f$ by $-f$).

We claim that the set of points in $S_g(p)$ with trivial
stabilizer in $\Gamma$ is open and dense. Indeed, on the one hand the points in
$W_g^s(p)$ with trivial stabilizer form an open and dense
subset, thanks to (2) in Lemma \ref{lemma:linearisation} and
the openness of $W_g^s(p)\subset M$ and on the other hand the
stabilizer of any $z\in W_g^s(p)$ is equal to the stabilizer of
the point in $S_g(p)$ it represents, because the action of
$\Gamma$ preserves $f$ and the restriction of $f$ to the fibers
of the projection $W_g^s(p)\setminus\{p\}\to S_g(p)$ is
injective.

Let $\psi\in L_{g,K}(p)$. Since $\psi$ does not belong to
$\Gamma_p\subset L_g(p)$ and since the union of the sets
$\{\Omega(p,q)\}_{q\in\OO}$ is an open and dense subset of
$S_g(p)$, there exist $q,q'\in\OO$ (not necessarily
distinct) and a point $c\in \Omega(p,q)$ satisfying $\psi
c\in\Omega(p,q')$ and $\psi c\neq \gamma c$ for every
$\gamma\in\Gamma_p$.
By the previous claim, we can also assume that the stabilizer of $c$ is trivial.

Choose a metric on the sphere $S_g(p)$. For any $y\in S_g(p)$ and
any $r>0$ denote by $\ov{B}_{S,r}(y)$ the closed ball in $S_g(p)$ of radius $r$ and center
$y$. Choose $\epsilon>0$ in such a way
that $\gamma \ov{B}_{S,\epsilon}(c) \cap \psi
\ov{B}_{S,\epsilon}(c)=\emptyset$ for every $\gamma\in\Gamma_p
$, and in such a way that $\ov{B}_{S,\epsilon}(c)\subset
\Omega_g(p,q)$ and $\psi \ov{B}_{S,\epsilon}(c)\subset
\Omega_g(p,q')$. Replacing $\epsilon$ by a smaller number if
necessary we can assume that the stabilizers of the points in
$\ov{B}_{S,\epsilon}(c)$ are all trivial.
Take $r$ distinct points $z_{\psi 1},\dots,z_{\psi
r}\in\ov{B}_{S,\epsilon/2}(c)$ and tangent vectors $u_{\psi
i}\in T_{\sigma_g^{p,q}(z_{\psi i})}S_g(q)$ for $i=1,\dots,r$.
Letting $z_{\psi}=(z_{\psi 1},\dots,z_{\psi r})$, we may assume
that $\sigma_g^{p,q}(z_{\psi})\in F_g(q)$ and
$\sigma_g^{p,q'}(\psi z_{\psi})\in F_g(q')$ because $F_g(q)$
(resp. $F_g(q')$) is dense in $S_g^r(q)$ (resp. $S_g^r(q')$).

Let $\oO_{\psi}$ be the set of all elements $\psi'\in L_g(p)$ satisfying
$\gamma \ov{B}_{S,\epsilon/2}(c)\cap
\psi'\ov{B}_{S,\epsilon/2}(c)=\emptyset$ for every
$\gamma\in\Gamma_p$ and $\sigma_g^{p,q'}(\psi' z_{\psi})\in
F_g(q')$.
%
Clearly $\oO_{\psi}\subset L_g(p)$ is open and contains $\psi$.


Denote the open ball in $M$ with center $x$ and radius $\delta$
by $B_{\delta}(x)$. Take real numbers $a<b<f(p)$ in such a way
that $[a,f(p))$ does not contain any critical value of $f$.
Take $\delta>0$ small enough so that $B_{\delta}(p)$ (resp.
$B_{\delta}(q)$, $B_{\delta}(q')$) is entirely contained in
$W_g^s(p)$ (resp. $W_g^u(q)$ and $W_g^u(q')$) and $\inf
f|_{B_{\delta}(p)}>b$ (resp. $\sup f|_{B_{\delta}(q)}<a$ and
$\sup f|_{B_{\delta}(q')}<a$).

Pick, for each $1\leq i\leq r$, points $x_i\in
B_{\delta}(p)\setminus\{p\}$ and $y_i\in
B_{\delta}(q)\setminus\{q\}$ both representing
$z_{\psi i}\in S_g(p)$, 
and a tangent vector $v_i\in T_{x_i}M$ projecting to $u_{\psi
i}\in T_{z_i}S_g(p)$. Define real numbers $t_1,\dots,t_r$ by
the condition that $y_i=\Phi_{t_i}^g(x_i)$, and let
$w_i=D\Phi_{t_i}^g(v_i)$.

Let $U\subset M$ be an open $\Gamma$-invariant subset contained
in $f^{-1}((a,b))\cap W_g^s(p)$ whose projection to $S_g(p)$
contains $z_{\psi 1},\dots,z_{\psi r}$ and is disjoint from
$\psi(\ov{B}_{S,\epsilon/2}(c))$. By Lemma
\ref{lemma:variacions-metrica-diff} one can pick a finite
dimensional vector subspace
$$G_{\psi}\subset\cC^{\infty}(M,S^2T^*M)^{\Gamma},$$
all of whose elements are supported in $U$, with the property that the linear map
\begin{equation}
\label{eq:surjective-map}
G_{\psi}\ni g'\mapsto
\left(\left.\frac{\partial}{\partial\epsilon}D\Phi^{g+\epsilon g'}_{-t_1}(w_1)\right|_{\epsilon=0},\dots,
\left.\frac{\partial}{\partial\epsilon}D\Phi^{g+\epsilon g'}_{-t_r}(w_r)\right|_{\epsilon=0}\right)
\in \bigoplus_{i=1}^r T_{v_i}(TM)
\end{equation}
is surjective.

Choose an open neighborhood $\oO_{\psi}'\subset\oO_{\psi}$ of $\psi$
whose closure in $\oO_{\psi}$ is compact.
Since $L_{g,K}(p)$ is compact there exist
$\psi_1,\dots,\psi_s\in L_{g,K}(p)$ such that
$L_{g,K}(p)\subset\oO'_{\psi_1}\cup\dots\cup\oO'_{\psi_s}$.
Denote $z_i=z_{\psi_i}\in S_g(p)^r$
and $u_i=u_{\psi_i}\in T(S_g(p)^r)$.

Let $G=\sum_iG_{\psi_i}$.
Let $\MMM$ be the set of all $g'\in\mM^{\Gamma}$ satisfying the following conditions:
\begin{enumerate}
\item $g'-g\in G$,
\item $\sigma_{g'}^{p,q}(z_i)\in F_{g'}(q)=F_g(q)$ for
    every $i$,
\item $\sigma_{g'}^{p,q'}(\psi z_i)\in F_{g'}(q')
    =F_{g}(q')$ for every $i$ and every $\psi\in
    \ov{\oO_{\psi_i}'}$.
\end{enumerate}
To explain conditions (2) and (3), note that since
$g'-g\in\sum_iG_{\psi_i}$ and the elements in each $G_{\psi_i}$
are supported away from the critical points, we can canonically
identify $S_g(q)=S_{g'}(q)$ and $S_g(q')=S_{g'}(q')$, and
similarly $F_g(q)=F_{g'}(q)$ and $F_g(q')=F_{g'}(q')$.

Note that $\{g'-g\mid g'\in\MMM\}$ can be identified with an
open subset of $G$ containing $0$, so $\MMM$ has a natural
structure of (finite dimensional) smooth manifold.


Consider, for each $i\in\{1,\dots,s\}$,
$$\VV_i=\{(g',\psi',b)\in \MMM\times \oO_{\psi_i}\times T(S_g(q')^r)\mid b=
D(\sigma_{g'}^{p,q'}\circ \alpha_{\psi'}\circ \sigma_{g'}^{q,p})(u_i)\}$$
and its subvariety
$$\VV_i'=\{(g',\psi',b)\in \MMM\times \oO'_{\psi_i}\times T(S_g(q')^r)\mid b=
D(\sigma_{g'}^{p,q'}\circ \alpha_{\psi'}\circ \sigma_{g'}^{q,p})(u_i)\}.$$
Let also
$$\AA=\MMM\times L_g(p)\times \aA_{g}(q')|_{F_g(q')}.$$
Note that $\VV_i,\,\VV_i',\,\AA$ are subvarieties of $\MMM\times L_g(p)\times T(S_g(q')^r)$.

Let $N_i=\VV_i\cap\AA\cap(\{g\}\times \oO_{\psi_i}\times T(S_g(p)^r)$.
The definition of $G_{\psi_i}$ guarantees that
$\VV_i$ and $\AA$ intersect transversely along $N_i$.
Consequently, there exists a neighborhood of $N_i$,
$$\nN_i\subset \MMM\times \oO_{\psi_i}\times T(S_g(q')^r),$$
such that
the intersection $\VV_i\cap\AA\cap \nN_i$ is a smooth manifold whose dimension satisfies
$$d-\dim(\VV_i\cap\AA\cap \nN_i)=\min\{d+1,(d-\dim\VV_i)+(d-\dim\AA)\},$$
where
$$d=\dim \MMM\times L_p(g)\times T(S_g(q')^r).$$
This formula is consistent with
the convention that a set is empty if and only if its dimension is $-1$.
Consider the projection
$$\pi_i:\VV_i\cap\AA\to \MMM.$$
Since the closure of $\VV_i'$ inside $\VV_i$ is compact, there exists a neighborhood of $g$, $\MMM_i\subset\MMM$, with the property that
$\pi_i^{-1}(\MMM_i)\cap\VV_i'\subset\nN_i$.
Hence,
$\pi_i^{-1}(\MMM_i)\cap\VV_i'$ is a smooth manifold. Let
$$\MMM_i^{\reg}\subset\MMM_i$$
be the set of regular values of $\pi_i$ restricted to $\pi_i^{-1}(\MMM_i)$.

We claim that for every $g'\in\MMM_i^{\reg}$ we have
$\pi_i^{-1}(g')\cap\VV_i'=\emptyset.$
To prove the claim it suffices to check that $\dim \pi_i^{-1}(g')\cap\VV_i'=-1$. Now,
\begin{align*}
\dim \pi_i^{-1}(g')\cap\VV_i' &= \dim(\VV_i\cap\AA\cap \nN_i) - \dim\MMM \\
&=d-\min\{d+1,(d-\dim\VV_i)+(d-\dim\AA)\}- \dim\MMM.
\end{align*}
If $(d-\dim\VV_i)+(d-\dim\AA)\geq d+1$ then this is clearly negative. So assume that instead
$(d-\dim\VV_i)+(d-\dim\AA)<d+1$. Since the projection of $\VV_i$ to $\MMM\times\oO_{\psi_i}$
is a diffeomorphism, we have
$$d-\dim\VV_i=d-(\dim \MMM\times\oO_{\psi_i})=d-(\dim\MMM\times L_g(p))=\dim T(S_g(q')^r)=2r(n-1).$$
On the other hand we have, using (2) in Lemma \ref{lemma:restriction-aA},
\begin{align*}
d-\dim\AA &= \dim T(S_g(q')^r)-\dim\aA_g(q')|_{F_g(q')} \\
&\geq 2r(n-1)-(r(n-1)+n^2-1)=r(n-1)-n^2+1.
\end{align*}
Combining both estimates we compute:
\begin{align*}
\dim \pi_i^{-1}(g')\cap\VV_i' &\leq d-2r(n-1)-r(n-1)+n^2-1-\dim\MMM \\
&=\dim L_p(g)\times T(S_g(q')^r)-3r(n-1)+n^2-1 \\
&\leq n^2+2r(n-1)-3r(n-1)+n^2-1 \\
&=2n^2-r(n-1)-1.
\end{align*}
Our choice of $r$, see (\ref{eq:def-r}),
implies that $2n^2-r(n-1)-1<0$, so the claim is proved.

Finally, let
$$\MMM^{\reg}=\MMM_1^{\reg}\cap\dots\cap\MMM_s^{\reg}.$$
We claim that $\MMM^{\reg}\subset \mM_{1,K}(p)$. Indeed,
suppose that $g'\in\MMM^{\reg}$ and let $\psi\in L_{g,K}(p)$ be
any element. Then $\psi\in\oO'_{\psi_i}$ for some $i$ and we
have, on the one hand,
$$z_i\in S_{g'}(p,q)^r,\quad \psi z_i\in S_{g'}(p,q')^r, \quad
\sigma_{g'}^{p,q}(z_i)\in F_{g'}(q),\quad
\sigma_{g'}^{p,q'}(\psi z_i)\in F_{g'}(q'),$$ and, on the other hand,
 the fact that
$\pi_i^{-1}(g')\cap\VV_i'=\emptyset$ implies that
$$D(\sigma_{g'}^{p,q'}\circ \alpha_{\psi}\circ
\sigma_{g'}^{q,p})(u_i)\notin\aA_{g'}(q')(\sigma_{g'}^{p,q'}\circ
\alpha_{\psi'}(z_i)).$$ This proves the claim.

Sard's theorem (see e.g. \cite[Chap 3, \S 1.3]{H}) implies that
$\MMM^{\reg}$ is residual in $\MMM$. Hence $\MMM^{\reg}$ is
dense in a neighborhood of $g\in\MMM$, so $\mM_{1,K}$ is dense
in a neighborhood of $g$.
\end{proof}

Recall that $\mM_{1,K}=\bigcap_{p\in\EE}\mM_{1,K}(p)$. The
preceding two lemmas imply:

\begin{lemma}
\label{lemma:mM-2-K-open-dense}
$\mM_{1,K}$ is a dense an open subset of $\mM_0$.
\end{lemma}

\section{Proof of Theorem \ref{thm:main} for $\dim M>1$}
\label{s:proof-thm:main}

Continuing with the notation of the previous sections, let us define
$$\mM_f=\mM_0\cap\bigcap_{K\in\NN}\mM_{1,K}.$$
Since each of the sets appearing in the right hand
side of the equality is open and dense in $\mM^{\Gamma}$  (see Subsection \ref{ss:mM-0} and Lemma \ref{lemma:mM-2-K-open-dense}), $\mM_f$ is a residual subset of
$\mM^{\Gamma}$.
Fix some $g\in\mM_f$ and let $\phi\in\Aut(\nabla^gf)$. We are going to check that there exists
some $\gamma\in\Gamma$ and some $t\in\RR$ such that
$$\phi(x)=\Phi^g_t(\gamma\,x)$$
for every $x\in M$. This will prove Theorem \ref{thm:main}.


\begin{lemma}
\label{lemma:main-M-0} For each $p\in\II$ (resp. $p\in\OO$)
there exists some $\gamma\in\Gamma$ and some $t\in\RR$
such that $\phi(x)=\Phi^g_t(\gamma\,x)$ for every $x\in W_g^s(p)$
(resp. for every $x\in W_g^u(p)$).
\end{lemma}
\begin{proof}
Suppose that $p\in\II$ (the case $p\in\OO$ is dealt with in the same way
with the obvious modificatoins). By property (C2) in the definition of
$\mM_0$ (see Subsection \ref{ss:mM-0}) there exists some $\gamma\in\Gamma$
such that $\phi(p)=\gamma\,p$. Hence, up to composing $\phi$ with the action of $\gamma$,
we can (and do) suppose that $\phi(p)=p$.

Once we know that $\phi$ fixes $p$, we conclude that it
restricts to a diffeomorphism of $W_g^s(p)$ preserving
$\nabla^gf$, which we identify with an element $\phi_p\in
L_g(p)$ via the isomorphism (\ref{eq:Aut-L}). Next, let us
prove that the action of $\phi_p$ on $S_g(p)$ coincides with
the action of some $\gamma\in\Gamma_p$. If this is not the
case, then $\phi_p\in L_{g,K}(p)$ for some natural $K$ (see
Subsection \ref{ss:def-mM-2}). Since $g\in\mM_{1,K}$, it
follows that there exist sources $q,q'\in\OO$ and
$z\in\Omega(p,q)^r$ satisfying
$$
\phi_p z\in\Omega_g(p,q')^r,\qquad \sigma_g^{p,q}z\in F_g(q),
\qquad \sigma_g^{p,q'}\phi_p z\in F_g(q'),$$ and a vector
$u\in\aA_g(q)(\sigma_g^{p,q}(z))$ satisfying
\begin{equation}
\label{eq:contra-1}
D(\sigma_g^{p,q'}\circ\alpha_{\phi_p}\circ\sigma_g^{q,p})(u)\notin
\aA_g(q')(\sigma_g^{p,q'}\circ\alpha_{\phi_p}(z)).
\end{equation}
By the definition of $\aA_g(q)$, we may write
$u=\yY_{g,s}(\sigma_g^{p,q}z)$ for some $s\in\Lie L_g(q)$.

The fact that $z\in\Omega(p,q)^r$ and $\phi_p
z\in\Omega_g(p,q')^r$ implies that $\phi(q)=q'$, so $\phi$ maps
$W_g^u(q)$ diffeomorphically to $W_g^u(q')$; since $\phi$
preserves $\nabla^gf$, $\phi$ induces by conjugation an
isomorphism
$$\psi:L_g(q)\to L_g(q').$$
The corresponding map at the level of Lie algebras associates
to $s$ an element $\psi(s)\in\Lie L_g(q')$, and in fact we have
\begin{align}
D(\sigma_g^{p,q'}\circ\alpha_{\phi_p}\circ\sigma_g^{q,p})(u) &=
D(\sigma_g^{p,q'}\circ\alpha_{\phi_p}\circ\sigma_g^{q,p})(\yY_{g,s}(\sigma_g^{p,q}z)) \notag \\
&=\yY_{g,\psi(s)}(\sigma_g^{p,q'}(\phi_p z)).
\end{align}
The last expression manifestly belongs to
$\aA_g(q')(\sigma_g^{p,q'}\circ\alpha_{\phi_p}(z))$, and this
contradicts (\ref{eq:contra-1}). So we have proved that there
is some $\gamma\in\Gamma_p$ such that $\gamma^{-1}\phi_p$ acts
trivially on $S_g(p)$. Now statement (1) in Lemma
\ref{lemma:restriction-aA} implies that
$\gamma^{-1}\phi_p=e^{t\xX_g(p)}$ for some $t\in\RR$, so we may
write $\phi_p=\gamma e^{t\xX_g(p)}$ or, equivalently, that
$\phi_p(y)=\Phi_t^g(\gamma\,y)$ for every $y\in W_g^s(p)$.
\end{proof}

For any $p\in\EE$ we denote $W_g(p):=W_g^s(p)$ (resp.
$W_g(p):=W_g^u(p)$) if $p\in\II$ (resp. if $p\in\OO$). Now the
proof of the case $\dim M>1$ in Theorem \ref{thm:main} is
concluded as the proof of the main theorem in \cite{TV}. This
is done in two steps. We know there exist
$\{t_p\in\RR\}_{p\in\EE}$ and $\{\gamma_p\in\Gamma\}_{p\in\EE}$
such that $\phi(x)=\Phi_{t_p}^g(\gamma_p x)$ for every
$p\in\EE$ and $x\in W_g(p)$. The first step consists in proving
that if all $\gamma_p$'s are equal then all $t_p$'s are equal
as well (this is \cite[Lemma 5]{TV}). The second step consists
on reducing the general case to the one covered by the first
step. This is explained in the three paragraphs following
\cite[Lemma 5]{TV}.

\appendix

\section{Change of the gradient flow as the metric varies}

Recall what we want to prove.

\begin{lemma}
\label{lemma:variacions-metrica-diff-2} Suppose that
$g\in\mM^{\Gamma}$, $x\in M\setminus\Crit(f)$ and
$y=\Phi^g_t(x)$ for some nonzero $t$. Suppose that the
stabilizer $\Gamma_x$ is trivial. Let $v\in T_xM$ be a nonzero
vector, and let $w=D\Phi^g_t(x)(v)$. Given any $u\in T_v(TM)$
and any $\Gamma$-invariant open subset $U\subset M$ containing
$\Phi^g_{t'}(x)$ for some $t'\in (0,t)$, there exists some
$g'\in\cC^{\infty}(M,S^2T^*M)^{\Gamma}$ supported on $U$ such
that
\begin{equation}
\label{eq:propietat-g-titlla}
\left.\frac{\partial}{\partial\epsilon}D\Phi^{g+\epsilon g'}_{-t}(w)\right|_{\epsilon=0}=u.
\end{equation}
\end{lemma}

We will prove Lemma \ref{lemma:variacions-metrica-2} using the following weaker version of it.

\begin{lemma}
\label{lemma:variacions-metrica-2}
Let $g,x,t,y$ be as in Lemma \ref{lemma:variacions-metrica-diff-2}.
Given any $v\in T_xM$ and any $\Gamma$-invariant open subset
$U\subset M$ containing $\Phi^g_{t'}(x)$ for some $t'\in
(0,t)$, there exists some
$g'\in\cC^{\infty}(M,S^2T^*M)^{\Gamma}$ supported on $U$ such
that
$$\left.\frac{\partial}{\partial\epsilon}\Phi^{g+\epsilon g'}_{-t}(y)\right|_{\epsilon=0}=v.$$
\end{lemma}

Before proving Lemma \ref{lemma:variacions-metrica-2} we prove two auxiliary lemmas.

\begin{lemma}
\label{lemma:metriques-adaptades}
Let $g\in\mM^{\Gamma}$.
Let $Y\in\cC^{\infty}(M;TM)^{\Gamma}$ satisfy $\supp Y\cap\Crit(f)=\emptyset$.
There exists some $\delta>0$ and a smooth map $G:(-\delta,\delta)\to\mM^{\Gamma}$ such that
$G(0)=g$ and, for every $\epsilon\in(-\delta,\delta)$,
$\nabla^{G(\epsilon)}f=\nabla^gf+\epsilon Y.$
\end{lemma}
\begin{proof}
This is a consequence of the following elementary fact in linear algebra. Let
$V$ be a finite dimensional real vector space and let $\alpha\in V^*$ be a nonzero element.
Let $\eE\subset S^2V^*$ be the open subset of Euclidean pairings, and let
$\nabla:\eE\to V$ be the map defined by the property that $e(\nabla(e),u)=\alpha(u)$
for every $u\in V$. Then $\nabla$ is a submersion with contractible fibers.
\end{proof}

For any vector field $X$ on $M$ we denote by $\Phi^X_t:M\to M$
the flow at time $t$ of $X$.

\begin{lemma}
\label{lemma:variacio-transport}
Let $X,Y\in\cC^{\infty}(M;TM)$. Suppose that $p\notin\supp Y$ and that $X$ has no regular periodic integral curve. For any $t$ we have
$$\left.\frac{\partial}{\partial s}\Phi^{X+sL_XY}_t(p)\right|_{s=0}=Y(\Phi^X_t(p)).$$
\end{lemma}
\begin{proof}
If $X(p)=0$ then the formula is immediate. So suppose that $X(p)\neq 0$.
Then $Z:=\{\Phi^X_{\tau}\mid 0\leq\tau\leq t\}$ is diffeomorphic to $[0,1]$, because
$X$ has no regular periodic integral curve. Take an open neighborhood $U\subset M$ of $Z$ and
coordinates $x=(x_1,\dots,x_n):U\to\RR^n$ with respect to which $p=(0,\dots,0)$
and $X=\frac{\partial}{\partial x_1}$, so that $\Phi^X_\tau(x_1,\dots,x_n)=(x_1+\tau,x_2,\dots,x_n)$.
Suppose that $x_*(L_XY|_U)=\sum a_j\frac{\partial}{\partial x_j}$ and that
$x_*(Y|_U)=\sum b_j\frac{\partial}{\partial x_j}$. Since $Y(p)=0$, we have
$b_j(t,0,\dots,0)=\int_0^ta_j(\tau,0,\dots,0)\,d\tau$ for every $j$.
Let $\gamma(t,s)=x(\Phi^{X+sL_XY}_t(p))$.
Let $e_1,\dots,e_n$ denote the canonical basis of $\RR^n$. We have
\begin{align*}
\left.\frac{\partial}{\partial t}\frac{\partial\gamma(t,s)}{\partial s}\right|_{s=0} &=
\left.\frac{\partial}{\partial s}\frac{\partial\gamma(t,s)}{\partial t}\right|_{s=0}=
\left.\frac{\partial}{\partial s}\left((1,0,\dots,0)+s\sum a_j(\gamma(t,s))\,e_j\right)\right|_{s=0} \\
&=\sum a_j(\gamma(t,0)) e_j=\sum a_j(t,0,\dots,0)\,e_j.
\end{align*}
Consequently, $$\left.\frac{\partial\gamma(t,s)}{\partial s}\right|_{s=0}=
\sum\left(\int_0^ta_j(\tau,0,\dots,0)\,d\tau\right)\,e_j=
\sum b_j(t,0,\dots,0)\,e_j=
x_*(Y(\Phi^X_t(p))),$$
which proves the desired formula.
\end{proof}

Let us now prove Lemma \ref{lemma:variacions-metrica-2}.
Fix some $g\in\mM^{\Gamma}$, let $x\in M\setminus\Crit(f)$, and
let $y=\Phi^g_t(x)$ for some nonzero $t$. Suppose that the
stabilizer $\Gamma_x$ is trivial, let $v\in T_xM$ be any
element and let $U\subset M$ be a $\Gamma$-invariant open
subset containing $\Phi^g_{t'}(x)$ for some $t'\in (0,t)$. Let
$X=\nabla^gf$. Since $x\notin\Crit(f)$ and $\Gamma_x$ is
trivial, there exists an invariant vector field
$Y\in\cC^{\infty}(M;TM)^{\Gamma}$ whose support is contained in
$U\setminus\Crit(f)$ and which satisfies $Y(x)=v$. By Lemma
\ref{lemma:variacio-transport} we have
$$\left.\frac{\partial}{\partial s}\Phi^{X+sL_XY}_{-t}(y)\right|_{s=0}=Y(x).$$
By Lemma \ref{lemma:metriques-adaptades} there exists some
$\delta>0$ and a family of metrics
$\{g_{\epsilon}\}_{\epsilon\in(-\delta,\delta)}$ satisfying
$g_0=g$ and $\nabla^{g_\epsilon}f=\nabla^gf+\epsilon Y$ for
every $\epsilon\in(-\delta,\delta)$. Setting $g'=\partial
g_{\epsilon}/\partial\epsilon|_{\epsilon=0}$ it follows that
$$\left.\frac{\partial}{\partial\epsilon}\Phi^{g+\epsilon g'}_{-t}(y)\right|_{\epsilon=0}=v.$$

Finally we prove Lemma \ref{lemma:variacions-metrica-diff}/\ref{lemma:variacions-metrica-diff-2}.

Let $\pi:TM\to M$ denote the projection and let $D\pi:T(TM)\to
TM$ denote its derivative. A vector field $X$ on $M$ defines a
vector field $\wt{X}$ on $TM$ by the condition that
\begin{equation}
\label{eq:def-titlla}
D\Phi_t^X=\Phi_t^{\wt{X}}
\end{equation}
for every $t$. In particular, $\Phi_{-t}^{\wt{X}}(w)=v$.

We will use these properties of the map
$X\mapsto\wt{X}$: (1) it is linear, (2)
$D\pi(\wt{X}(u))=X(\pi(u))$ for every vector field $X$ and any
$u\in TM$, (3) if $X(a)=0$ for some point $a\in M$, then the
restriction of $\wt{X}$ to $T_aM$ is vertical and can be
identified with the linear vector field $T_aM\to T_aM$ given by
the endomorphism $DX(a)$ of $T_aM$, and (4) it is compatible
with Lie brackets: $\wt{[X,Y]}=[\wt{X},\wt{Y}]$.

Let $g\in\mM^{\Gamma}$, let $x\in M\setminus\Crit(f)$ be a
point with trivial stabilizer, and let $y=\Phi^g_t(x)$ for some
nonzero $t$. Let $v\in T_xM$ be nonzero and let
$w=D\Phi^g_t(x)(v)$. Finally, suppose given $u\in T_v(TM)$ and
a $\Gamma$-invariant open subset $U\subset M$ containing
$\Phi^g_{t'}(x)$ for some $t'\in (0,t)$.

Let $X=\nabla^gf$. Since $x\notin\Crit(f)$ and $\Gamma_x$ is
trivial, there exists an invariant vector field
$Y_0\in\cC^{\infty}(M;TM)^{\Gamma}$ whose support is contained
in $U\setminus\Crit(f)$ and which satisfies $Y_0(x)=D\pi(u)$.
Then $u_1=\wt{Y_0}(v)-u$ satisfies $D\pi(u_1)=0$. Let
$L:T_xM\to T_xM$ be a linear map satisfying $Lv=u_1$, and let
$Y_1\in\cC^{\infty}(M;TM)^{\Gamma}$ have support contained in
$U\setminus\Crit(f)$ and satisfy $Y_1(x)=0$ and $DY_1(x)=L$.
Let $Y=Y_0+Y_1$. Then $\wt{Y}(v)=u$. By (\ref{eq:def-titlla}), the properties
of the map $X\mapsto\wt{X}$, and Lemma
\ref{lemma:variacio-transport} (with $M$ replaced by $TM$ in the
statement), we have
$$\left.\frac{\partial}{\partial
s}D\Phi^{X+sL_XY}_{-t}(w)\right|_{s=0} = \left.\frac{\partial}{\partial
s}\Phi^{\wt{X+sL_XY}}_{-t}(w)\right|_{s=0}
=\left.\frac{\partial}{\partial s}\Phi^{\wt{X}+sL_{\wt{X}}\wt{Y}}_{-t}(w)\right|_{s=0}
=\wt{Y}(v).$$ By Lemma
\ref{lemma:metriques-adaptades} there exists some $\delta>0$
and a smooth map $G:(-\delta,\delta)\to\mM^{\Gamma}$
satisfying
$G(0)=g$ and $\nabla^{G(\epsilon)}f=\nabla^gf+\epsilon Y$ for
every $\epsilon\in(-\delta,\delta)$. Then $g'=G'(0)$ satisfies (\ref{eq:propietat-g-titlla}).

\section{Glossary}

Here we list in alphabetical order some of the symbols used in Section \ref{s:sternberg} and the next ones.
\begin{align*}
\aA_g(p) =& \,\,\text{the distribution on $S_g(p)^r$ given by the infinitesimal diagonal action} \\
&\,\,\text{of the Lie algebra of $L_g(p)$, see Subsection \ref{ss:singular-distributions}}\\
\alpha_\psi =& \,\,\text{the map $S_g(p)^r\to S_g(p)^r$ given by the action of $\psi\in L_g(p)$}\\
\EE =& \,\,\II\cup\OO \\
f:M\to\RR =& \,\,\text{a $\Gamma$-invariant Morse function on $M$, see Subsection \ref{ss:sink,sources}} \\
F_g(p) =& \,\,\text{a dense open subset of $S_g(p)^r$ on which the only elements of $L_g(p)$ with} \\
&\,\,\text{fixed points are those of the form $D\Phi_g^t(p)$, see Subsection \ref{ss:F-g-q}} \\
\Phi_g^t =& \,\,\text{the time $t$ gradient flow of $f$ w.r.t. the metric $g$}\\
\Gamma =& \,\,\text{a finite group acting smoothly and effectively on $M$}\\
\II =& \,\,\text{the critical points of $f$  of index $n=\dim M$ (sinks of $\nabla^gf$)} \\
L_g(p) =& \,\,\text{the automorphisms of $T_pM$ commuting with $D\nabla^gf(p)$, see Subsection
\ref{ss:sink,sources};} \\
&\,\,\text{if $p\in\II$ and $g\in\mM_0$ then $L_g(p)$ is naturally isomorphic to
$\Aut(\nabla^gf|_{W_g^s(p)})$}\\
&\,\,\text{(if $p\in\OO$ then the same holds for $W_g^u(p)$), see Subsection \ref{ss:mM-0}} \\
M =& \,\,\text{a compact connected smooth manifold of dimension at least $2$}\\
\mM_0 =& \,\,\text{the set of $\Gamma$-invariant metrics on $M$ defined in Subsection \ref{ss:mM-0}}\\
\mM_{1,K} =& \,\,\text{the set of $\Gamma$-invariant metrics on $M$ defined in Section \ref{s:mM-2-K}}\\
\OO =& \,\,\text{the critical points of $f$ of index $0$ (sources of $\nabla^gf$)} \\
\Omega_g(p,q) =& \,\,\text{the projection to $S_g(q)$ of $W_g^s(p)\cap W_g^u(q)$ if $p\in\II$ and $q\in \OO$, and the} \\
&\,\,\text{projection of $W_g^u(p)\cap W_g^s(q)$ if $p\in\OO$ and $q\in \II$, see Subsection \ref{ss:spheres}}\\
S_g(p) =& \,\,\text{the set of nonconstant integral curves of $\nabla^gf|_{W^s_g(p)}$
(resp. $\nabla^gf|_{W^u_g(p)}$)}\\ 
&\,\,\text{if $p\in\II$ (resp. $p\in\OO$), see Subsection \ref{ss:spheres}}\\
\sigma_g^{p,q} =& \,\,\text{the natural isomorphism $\Omega_g(p,q)\to\Omega_g(q,p)$, see Subsection \ref{ss:spheres}} \\
W_g^{s}(p) = &\,\,\text{the stable set of $p\in\II$} \\
W_g^{u}(q) = &\,\,\text{the unstable set of $q\in\OO$} \\
\end{align*}

\end{document}